\documentclass[11pt]{amsart}

\usepackage[utf8]{inputenc}
\usepackage[T1]{fontenc}
\usepackage{lmodern}
\usepackage{enumitem}
\usepackage{fullpage}
\usepackage{color}
\usepackage{hyperref}
\usepackage{stmaryrd, amsthm, mathtools, amssymb, dsfont, mathrsfs, bm}
\usepackage{booktabs,diagbox}
\usepackage[noabbrev,capitalise]{cleveref}
\usepackage{algorithm}
\usepackage{algpseudocode}
\usepackage{tikz}
\usetikzlibrary{decorations.pathreplacing,positioning}
\usepackage{caption}
\usepackage{subcaption}

\newtheorem{thm}{Theorem}
\newtheorem{prop}{Proposition}
\newtheorem{cor}{Corollary}
\newtheorem{conj}{Conjecture}
\newtheorem{lemma}{Lemma}
\theoremstyle{definition}
\newtheorem{defi}{Definition}
\crefname{defi}{Definition}{Definitions}
\theoremstyle{remark}
\newtheorem{rk}{Remark}
\newcounter{cas}

\newtheoremstyle{assert}
  {.5\baselineskip±.2\baselineskip}   
  {.5\baselineskip±.2\baselineskip}   
  {\itshape}  
  {0pt}       
  {\bfseries} 
  {.}         
  {5pt plus 1pt minus 1pt} 
  {(\thmnumber{#2})}          
\theoremstyle{assert}
\newtheorem{as}[cas]{}

\DeclareMathOperator{\di}{dist}
\DeclareMathOperator{\girth}{girth}
\DeclareMathOperator{\et}{and}

\newcommand{\eps}{\varepsilon}
\newcommand{\ba}{\bm{\alpha}}
\newcommand{\be}{\mathbf{e}}
\newcommand{\bI}{\mathbf{I}}
\newcommand{\bX}{\mathbf{X}}

\newcommand{\cI}{\mathscr{I}}
\newcommand{\cIm}{\cI_{\max}}
\newcommand{\cIM}{\cI_{\alpha}}

\newcommand{\pth}[1]{\left( #1 \right)}

\newcommand{\pr}[1]{\mathbb{P}\left[ #1 \right]}
\newcommand{\esp}[1]{\mathbb{E}\left[ #1 \right]}
\newcommand{\sst}[2]{\left\{#1\,:\,#2\right\}}
\newcommand{\abs}[1]{\left\lvert#1\right\rvert}

\newcommand{\val}{\iota}

\tikzstyle{vertex} = [draw,fill,shape=circle,node distance=80pt]
\tikzstyle{wertex} = [draw=black,fill=white,shape=circle,node distance=80pt]
\tikzstyle{gertex} = [draw=black,fill=black!25,shape=circle,node distance=80pt]
\tikzstyle{edge} = [fill,opacity=.5,fill opacity=.5,line cap=round, line join=round, line width=50pt]
\pgfdeclarelayer{background}
\pgfsetlayers{background,main}

\makeatletter
\newcommand{\newpar}[1]{%
    \par
    \addvspace{\medskipamount}
    \noindent\textit{#1\@addpunct{.}}\enspace\ignorespaces
    }
\makeatother

\title{Fractional chromatic number, maximum degree and girth}
\author{François Pirot}
\address{\'Equipe Orpailleur, LORIA (Université de Lorraine, C.N.R.S., INRIA),
Vandœuvre-lès-Nancy, France and
Department of Mathematics, Radboud University Nijmegen, Netherlands.}
\email{francois.pirot@loria.fr}

\author{Jean-Sébastien Sereni}
\address{Service public français de la recherche, Centre National de la Recherche Scientifique (ICube, CSTB), Strasbourg, France.}
\email{sereni@kam.mff.cuni.cz}

\begin{document}

\begin{abstract}
We introduce a new method for computing bounds on the independence
number and fractional chromatic number of classes of graphs with local
constraints, and apply this method in various scenarios. 
We establish a formula that generates a general upper bound for the fractional chromatic number of
triangle-free graphs of maximum degree~$\Delta \ge 3$. This upper bound matches that deduced from the
fractional version of Reed's bound for small values of~$\Delta$, and improves it when~$\Delta\ge 17$,
transitioning smoothly to the best possible asymptotic regime, barring a breakthrough in Ramsey theory. 
Focusing on smaller values of~$\Delta$, we also demonstrate that every
    graph of girth at least~$7$ and maximum degree~$\Delta$ has fractional
    chromatic number at most~$1+ \min_{k \in \mathbb{N}} \frac{2\Delta +
    2^{k-3}}{k}$. In particular, the fractional chromatic number of a graph
    of girth~$7$ and maximum degree~$\Delta$ is
    at most~$\frac{2\Delta+9}{5}$ when~$\Delta \in [3,8]$,
    at most~$\frac{\Delta+7}{3}$ when~$\Delta \in [8,20]$,
    at most~$\frac{2\Delta+23}{7}$ when~$\Delta \in [20,48]$, and
    at most~$\frac{\Delta}{4}+5$ when~$\Delta \in [48,112]$.
In addition, we also obtain new lower bounds on the independence ratio of graphs of maximum
    degree~$\Delta \in \{3,4,5\}$ and girth~$g\in \{6,\dotsc,12\}$,
    notably~$1/3$ when~$(\Delta,g)=(4,10)$ and~$2/7$ when~$(\Delta,g)=(5,8)$.
\end{abstract}

\maketitle

\vspace{-8pt}
\section{Introduction} 

\subsection{A motivation coming from Ramsey theory}
Since the seminal result of Ramsey~\cite{Ram30},
the branch of combinatorics now called ``Ramsey Theory'' has known an ever increasing range of interest from the community. 

\begin{thm}[Ramsey, 1930]
For every integers~$s,t\ge 2$, there exists a minimal integer~$R(s,t)$ such that every graph on~$n\ge R(s,t)$ vertices contains either a clique of size~$s$, or an independent set of size~$t$.
\end{thm}

Computing the exact value of the \emph{Ramsey number}~$R(s,t)$ for all possible pairs~$(s,t)$ is a notorious problem, and only little progress has been made since the first quantitative result due to Erd\H os and Szekeres~\cite{ErSz35} that~$R(s,t) \le \binom{s+t-2}{s-1}$ for every~$s,t\ge 2$. Two particular regimes of the Ramsey numbers have attracted a particular focus, namely the \emph{diagonal Ramsey numbers}~$R(s,s)$, and the \emph{off-diagonal Ramsey numbers}~$R(s,t)$ where~$s$ is a fixed constant (typically~$s=3$) and~$t\to \infty$. The result of Erd\H os and Szekeres~\cite{ErSz35} implies that~$\frac{1}{s} \ln R(s,s) \le 4$, while Erd\H os~\cite{Erd47} showed in 1947 through an analysis of random graphs drawn from~$G(n,\frac{1}{2})$ that~$\frac{1}{s} \ln R(s,s) \ge \frac{1}{2}$. Reducing this gap is still an open problem from these days.
Our work is mainly motivated by the off-diagonal regime. In this setting, it is relevant to introduce the \emph{maximum degree}~$\Delta(\cdot)$ as an additional parameter, since on one hand~$\Delta(G)$ is an direct lower bound on the independence number of a triangle-free graph~$G$, and on the other hand graphs with smaller maximum degree are easier to properly colour, and hence contain larger independent sets. The best general upper bound on~$R(3,t)$ to this date is due to Shearer~\cite{She83} and can be deduced from the following result.

\begin{thm}[Shearer, 1983]
\label{thm:shearer}
Every triangle-free graph~$G$ on~$n$ vertices and of average degree~$d$ contains an independent set of size at least~$\frac{d\ln d - d + 1}{(d-1)^2}n$. 
In particular, this implies that~$R(3,t)\lesssim \frac{t^2}{\ln t}$.
\end{thm}

In this paper, we are interested in finding refinements of \Cref{thm:shearer}. We will study the \emph{fractional chromatic number} and \emph{Hall ratio} of graphs of given \emph{girth}. Before going further, we introduce the relevant notions.

\subsection{Definitions and observations}

The fractional chromatic number~$\chi_f(G)$ of a graph~$G$ is a refinement of
the chromatic number. It is the fractional solution to a linear program, the
integer solution of which is the chromatic number. Let~$G$ be a given graph; we
define~$\cI(G)$ to be the set of all independent sets of~$G$. We will often restrict to the set~$\cI_{\max}(G)$ of all maximal independent sets of~$G$, or to the set~$\cI_\alpha(G)$ of all maximum independent sets of~$G$.
Then the \emph{fractional chromatic number}~$\chi_f(G)$ of~$G$ is the solution of the following linear program.
\begin{align*}
 & \quad \min \sum_{I \in \cI(G)}  w_I\\
\text{such that} & \begin{cases}
        w_I \in [0,1] &\quad \text{for each~$I\in\cI(G)$}\\
    \displaystyle\sum_{\substack{I\in\cI(G)\\v\in I}} w_I \ge 1&\quad\text{for each~$v\in V(G)$}.
\end{cases}
\end{align*}

A \emph{fractional colouring of weight~$w$ of~$G$} is any instance within the domain of the above linear
program such that~$\sum w_I = w$. Observe that a fractional colouring of weight~$w$ remains valid after adding
vertices to its non-maximal independent sets, and that this does not affect its weight. Therefore, the optimal
value of the above linear program remains unchanged if one additionally requires that only maximal independent
sets, i.e. those in~$\cI_{\max}(G)$, are given a positive weight.  A~$k$-colouring of~$G$ is a special case of
a fractional colouring of weight~$k$ of~$G$, where~$w_I = 1$ if~$I$ is a monochromatic class of the
$k$-colouring, and~$w_I = 0$ otherwise. Note also that if~$W$ is a clique in~$G$, then any fractional
colouring of~$G$ is of weight at least~$\abs{W}$, and hence~$\chi_f(G) \ge \omega(G)$. More generally, if~$H$
is a subgraph of~$G$ then~$\chi_f(G)\ge\chi_f(H)$.
Another elementary lower bound on the fractional chromatic number comes from the independence
number~$\alpha(G)$. Indeed, the total weight induced by an independent set~$I$ of~$G$ on its vertex set is at
most~$w_I \alpha(G)$, and so the weight of a fractional colouring of~$G$ is at least~$\abs{V(G)}/\alpha(G)$. These
two lower bounds can be combined by the \emph{Hall ratio~$\rho(G)$} of~$G$, which is defined as~$\rho(G)
\coloneqq \max\sst{\abs{V(H)}/\alpha(H)}{H \subseteq G}$. The above observations allow us
to write the following inequalities:
\[ \max\left\{\omega(G),\frac{\abs{V(G)}}{\alpha(G)}\right\} \le \rho(G) \le
\chi_f(G) \le \chi(G) \le \Delta(G)+1, \]
where~$\Delta(G)$ is the maximum degree of~$G$. If~$G$ is a perfect graph then equality holds between~$\omega(G)$
and~$\chi(G)$, and so in particular between~$\omega(G)$ and~$\chi_f(G)$. Perfect graphs are those
graphs that contain no odd hole nor odd antihole, as was conjectured by Berge~\cite{Ber61} in~1961,
and proved by Chudnovsky \emph{et al.}~\cite{CRST06} in~2006. On the other side, the characterisation of the
graphs~$G$ for which equality holds between~$\chi(G)$ and~$\Delta(G)+1$ was established by Brooks~\cite{Bro41}
in~1941, and those graphs are cliques and odd cycles.  Since~$\chi_f(C_{2k+1}) = \frac{k}{2k+1}$, the only
graphs~$G$ such that~$\chi_f(G)=\Delta(G)+1$ are cliques.  Moreover, equality holds between the Hall ratio
of~$G$ and its fractional chromatic number for example when~$G$ is vertex transitive.

\subsection{Previous results on the Hall ratio}

Since the Hall ratio is the hereditary version of the inverse of the independence ratio (defined as the independence number divided by the number of vertices), any result on the independence ratio in a hereditary class of graphs can be extended to the Hall ratio. 
The Hall ratio of a graph has often been studied in relation with the \emph{girth}, which is
the length of a smallest cycle in the graph. A first result in this direction is the
celebrated introduction of the so-called ``deletion method'' in graph theory by Erd\H{o}s,
who used it to demonstrate the existence of graphs with arbitrarily large girth and
chromatic number. The latter is actually established by proving that the Hall ratio of the
graph is arbitrarily large. As a large girth is not strong enough a requirement to imply a
constant upper bound on the chromatic number, a way to pursue this line of research is to
express the upper bound in terms of the maximum degree~$\Delta(G)$ of the graph~$G$
considered. This also applies to the Hall ratio.

Letting~$\girth(G)$ stand for the girth of the graph~$G$, that is, the length of a shortest
cycle in~$G$ if~$G$ is not a forest and~$+\infty$ otherwise, we define~$\rho(d,g)$ to be the
supremum of the Hall ratios over all graphs of maximum degree at most~$d$ and girth at
least~$g$. We also let~$\rho(d,\infty)$ be the limit as~$g\to \infty$ of~$\rho(d,g)$ ---
note that if we fix~$d$ then~$\rho(d,g)$ is a non-increasing function of~$g$.
In symbols,~$\rho(d,g) \coloneqq \sup\sst{\abs{V(G)}/\alpha(G)}{\text{$G$ graph with~$\Delta(G)\le d$ and~$\girth(G)\ge g$}}$, and~$
\rho(d,\infty) \coloneqq \lim_{g \to \infty}\limits \rho(d,g)$.

In~1979, Staton~\cite{Sta79} established that~$\rho(d,4)\le\frac{5d-1}{5}$, in particular
implying that~$\rho(3,4)\le\frac{14}{5}$. The two graphs depicted in Figure~\ref{graph1},
called the graphs of Fajtlowicz and of Locke, have fourteen vertices each, girth~$5$, and no independent set of order~$6$. It follows that~$\rho(3,4)=\frac{14}{5}=\rho(3,5)$.
It is known that the graphs of Fajtlowicz and of Locke are the only two cubic
triangle-free and connected graphs with Hall ratio~$\frac{14}{5}$.
This follows from a result
of Fraughnaugh and Locke~\cite{FrLo95} for graphs with more than~$14$ vertices completed by an exhaustive computer check
on graphs with at most~$14$ vertices performed by Bajnok and Brinkmann~\cite{BaBr98}.

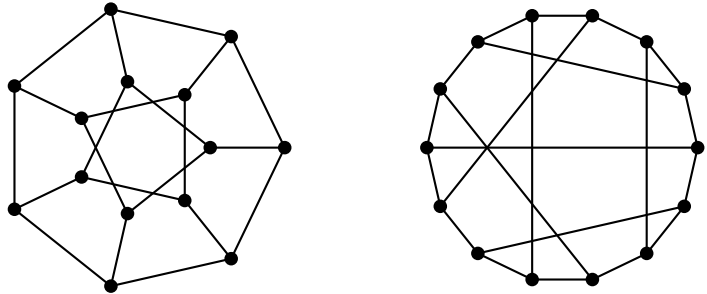
\begin{figure}[!ht]
\begin{center}
\begin{tikzpicture}[scale=0.9]
\def\a{360/7}
\foreach \k in {0, 1,...,6}{
	\draw[thick] (\k*\a:2.1) -- (\k*\a+360/7:2.1);
	\draw[thick] (\k*\a:1) -- (\k*\a+360*2/7:1);
	\draw[thick] (\k*\a:2.1) -- (\k*\a:1);
	\fill[black] (\k*\a:2.1) circle(0.1);
	\fill[black] (\k*\a:1) circle(0.1);
}
	
\tikzset{shift={(6.2,0)}}

\def\b{360/14}
\foreach \k in {0, 1,...,13}{
	\draw[thick] (\k*\b:2) -- (\k*\b+360/14:2);
	\fill[black] (\k*\b:2) circle(0.1);
}
	
	\draw[thick] (0:2) -- (180:2);
	\draw[thick] (180/7:2) -- (180*5/7:2);
	\draw[thick] (180*2/7:2) -- (180*12/7:2);
	\draw[thick] (180*3/7:2) -- (180*8/7:2);
	\draw[thick] (180*4/7:2) -- (180*10/7:2);
	\draw[thick] (180*6/7:2) -- (180*11/7:2);
	\draw[thick] (180*9/7:2) -- (180*13/7:2);
\end{tikzpicture}
\caption{The two cubic triangle-free connected graphs with Hall ratio~$\frac{14}{5}$.}\label{graph1}
\end{center}
\end{figure}

In~1983, Jones~\cite{Jon84} reached the next step by establishing that~$\rho(4,4)
= \frac{13}{4}$. Only one connected graph is known to attain this value: it has~$13$
vertices and is presented in Figure~\ref{graph2}. The value of~$\rho(d,4)$ when~$d \ge 5$ is still unknown;
the best general upper bound is due to Shearer~\cite{She91}, and improves his bound stated in
\Cref{thm:shearer}. He also provided an upper bound for~$\rho(d,6)$ as a consequence of a stronger result on
graphs with no cycle of length~$3$ or~$5$.

\begin{figure}[!ht]
\begin{center}
\begin{tikzpicture}[scale=0.9]
\def\a{360/13}
\foreach \k in {0, 1,...,12}{
	\draw[thick] (\k*\a:2) -- (\k*\a+360/13:2);
	\draw[thick] (\k*\a:2) -- (\k*\a+360*5/13:2);
	\fill[black] (\k*\a:2) circle(0.1);
	}
\end{tikzpicture}
\vspace{-18pt}
\caption{The only known 4-regular triangle-free connected graph of Hall ratio~$\frac{13}{4}$.}\label{graph2}
\end{center}
\end{figure}
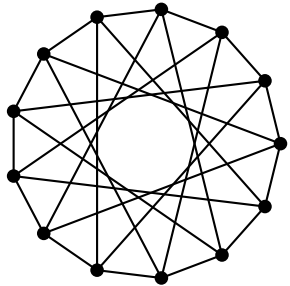

\begin{thm}[Shearer, 1991]\label{thm:shearer1}
Set~$f(0)\coloneqq 1$, and~$f(d) \coloneqq  \frac{1+(d^2-d)f(d-1)}{d^2+1}$ for every integer~$d\ge 2$.
If~$G$ is a triangle-free graph on~$n$ vertices with degree sequence~$d_1,\dotsc,d_n$, then~$G$ contains an independent of size~$\sum_{i=1}^n f(d_i)$.
\end{thm}

\begin{thm}[Shearer, 1991]\label{thm:shearer2}
Set~$f(0)\coloneqq 0$,~$f(1)\coloneqq \frac{4}{7}$, and~$f(d) \coloneqq \frac{1+(d^2-d)f(d-1)}{d^2+1}$ for every integer~$d\ge 2$.
If~$G$ is a graph on~$n$ vertices with degree sequence~$d_1,\dotsc,d_n$ and with
no~$3$-cycle and no~$5$-cycle, then~$G$ contains an independent of size~$\sum_{i=1}^n f(d_i) - \frac{n_{11}}{7}$,
where~$n_{11}$ is the number of pairs of adjacent vertices of degree~$1$ in~$G$.
\end{thm}

Theorems~\ref{thm:shearer1} and~\ref{thm:shearer2} allow us to compute upper bounds on~$\rho(d,4)$ and on~$\rho(d,6)$ for small values of~$d$, as indicated in Table~\ref{tab:1}. When~$d \ge 5$, these bounds are the best known ones.

\renewcommand{\arraystretch}{1.2}
\begin{table}[!ht]\centering
    \begin{tabular}{@{}crlrl@{}}\toprule
~$d$ & \multicolumn{2}{c}{upper bound on~$\rho(d,4)$} & \multicolumn{2}{c}{upper bound on~$\rho(d,6)$} \\
\midrule
$2$ &~$\frac{5}{2}$ &~$=2.5$ &~$\frac{7}{3}$ &$\approx 2.33333$ \\
$3$ &~$\frac{50}{17}$ &$\approx 2.94118$ &~$\frac{14}{5}$ &$= 2.8$ \\
$4$ &~$\frac{425}{127}$ &$\approx 3.34646$ &~$\frac{119}{37}$ &$\approx 3.21622$ \\
$5$ &~$\frac{2210}{593}$ &$\approx 3.72681$ &~$\frac{3094}{859}$ &$\approx 3.60186$ \\
$6$ &~$\frac{8177}{2000}$ &$\approx 4.0885$ &~$\frac{57239}{14432}$ &$\approx 3.96612$ \\
$7$ &~$\frac{408850}{92177}$ &$\approx 4.43549$ &~$\frac{408850}{94769}$ &$\approx 4.31417$ \\
$8$ &~$\frac{13287625}{2785381}$ &$\approx 4.77049$ &~$\frac{13287625}{2857957}$ &$\approx 4.64934$ \\
$9$ &~$\frac{1089585250}{213835057}$ &$\approx 5.09545$ &~$\frac{1089585250}{219060529}$ &$\approx 4.9739$ \\
$10$ &~$\frac{11004811025}{2033474038}$ &$\approx 5.41183$ &~$\frac{11004811025}{2080503286}$ &$\approx 5.28949$\\ \bottomrule
\end{tabular} 
\vspace{6pt}
    \caption{Upper bounds on~$\rho(d,4)$ and~$\rho(d,6)$ for~$d\le 10$ derived from Theorems~\ref{thm:shearer1} and~\ref{thm:shearer2}.}\label{tab:1}
\end{table}
\renewcommand{\arraystretch}{1}

We are not aware of any non-trivial lower bounds on~$\rho(5,4)$ and~$\rho(6,4)$.
Figure~\ref{graph34} show graphs illustrating that~$\rho(5,4) \ge
\frac{10}{3} \approx 3.33333$ and~$\rho(6,4) \ge \frac{29}{8}=3.625$. These two graphs
are circulant graphs, which are Cayley graphs over~$\mathbb{Z}_n$.

\begin{figure}[!ht]
\centering
\subcaptionbox{A~$5$-regular triangle-free (vertex-transitive) graph with Hall ratio~$\frac{10}{3}$.    It is the Cayley graph over~$\mathbb{Z}_{20}$ with generating set~$\{\pm1,\pm6,10\}$.
    There is no independent set of order~$7$, and the white vertices form an independent set of order~$6$.}[7.5cm]{
\begin{tikzpicture}[scale=.8]

\foreach \k in {9,...,19,1}{
	\draw[thick] (\k*18:3) -- (\k*18+360/20:3);
	\draw[thick] (\k*18:3) -- (\k*18+360*6/20:3);
	\fill[black] (\k*18:3) circle(0.1);
	}

\foreach \k in {0,2,...,12}{
	\draw[thick] (\k*18:3) -- (\k*18+360/20:3);
	\draw[thick] (\k*18:3) -- (\k*18+360*6/20:3);
	\fill[black] (\k*18:3) circle(0.1);
	}

\foreach \k in {0, 1,...,9}{
    \draw[thick] (18*\k:3) to (18*\k+180:3);
	}

\foreach \k in {3,5,7,14,16,18}{
	\draw[thick] (\k*18:3) -- (\k*18+360/20:3);
	\draw[thick] (\k*18:3) -- (\k*18+360*6/20:3);
	\fill[black] (\k*18:3) circle(0.1);
	\fill[white] (\k*18:3) circle(0.08);
	}
\end{tikzpicture}}\hspace*{40pt}
\subcaptionbox{A~$6$-regular triangle-free (vertex-transitive) graph with Hall ratio~$\frac{29}{8}$.
    It is the Cayley graph over~$\mathbb{Z}_{29}$ with generating set~$\{\pm1,\pm5,\pm13\}$.
    There is no independent set of order~$9$, and the white vertices form an independent set of order~$8$.}[7.5cm]{
\begin{tikzpicture}[rotate=-8,scale=.8]
\def\a{360/29}
\foreach \k in {0,1,...,28}{
	\draw[thick] (\k*\a:3) to (\k*\a+360/29:3);
	\draw[thick] (\k*\a:3) to[out=\k*\a+160,in=\k*\a+360*5/29+200] (\k*\a+360*5/29:3);
	\draw[thick] (\k*\a:3) -- (\k*\a+360*13/29:3);
	\fill[black] (\k*\a:3) circle(0.1);
	}

\foreach \k in {8,10,...,22}{
	\fill[white] (\k*\a:3) circle(0.08);
	}
\end{tikzpicture}}
\caption{Two possibly extremal regular triangle-free graphs for the Hall ratio.}\label{graph34}
\end{figure}

The value of~$\rho(3,g)$ has also been studied when~$g$ goes to infinity.
Kardoš, Král' and Volec~\cite{KKV11} proved the existence of an integer~$g_0$
such that~$\rho(3,g_0) \le 2.2978$. This result has been improved by Hoppen and Wormald~\cite{HoWo18} to~$\rho(3,g_0) \le 2.2854$. More strongly, these upper bounds hold for
the fractional chromatic number of every (sub)cubic graph of girth at
least~$g_0$. In the other direction, Bollobás~\cite{Bol81} proved a general
lower bound on~$\rho(d,g)$.

\begin{thm}[Bollobás, 1981]\label{thm-bol}
    Let~$d\ge3$. Let~$\alpha$ be a real number in~$(0,1)$ such that
    \[\alpha(d\ln 2-\ln(\alpha))+(2-\alpha)(d-1)\ln(2-\alpha)+(\alpha-1)d\ln(1-\alpha)<2(d-1)\ln 2.\]
    For every integer~$g$, there exists a~$d$-regular graph with
    girth at least~$g$ and Hall ratio more than~$2/\alpha$.
\end{thm}
Theorem~\ref{thm-bol} allows us to compute lower bounds on~$\rho(d, \infty)$
for any value of~$d$, the smaller ones being presented in Table~\ref{tab:2}.
All these values can be generalised into a looser but asymptotically equivalent
general lower bound of~$d/(2\ln d)$~\cite[Corollary~3]{Bol81}.

\begin{table}[!ht]\centering
\begin{tabular}{@{ }cc@{ }}
\toprule
$d$ & lower bound on~$\rho(d,\infty)$ \\
\midrule
$2$ &~$2$ \\
$3$ &~$2.17835$ \\
$4$ &~$2.3775$ \\
$5$ &~$2.57278$ \\
$6$ &~$2.76222$\\
$7$ &~$2.94606$\\
$8$ &~$3.1249$\\
$9$ &~$3.29931$ \\
$10$ &~$3.46981$ \\
$d$ &~$d/(2\ln d)$\\
\bottomrule
\end{tabular}
\vspace{6pt}
    \caption{Lower bounds on~$\rho(d,\infty)$ implied by Theorem~\ref{thm-bol}.}\label{tab:2}
\end{table}
\vspace{-12pt}

\subsection{Previous results on the fractional chromatic number}
Recently, Molloy~\cite{Mol19} proved the best known extremal upper bounds
for the chromatic number of graphs of given clique number and maximum degree.

\vbox{%
\begin{thm}[Molloy, 2019]\label{thm-molloy}
Let~$G$ be a graph of maximum degree~$\Delta$.
\begin{itemize}
\item If~$G$ is triangle-free, then for every~$\eps>0$, there exists~$\Delta_\eps$ such that, assuming that~$\Delta \ge \Delta_\eps$,
\[ \chi(G) \le (1+\eps)\frac{\Delta}{\ln \Delta}.\]
\item If~$G$ has clique number~$\omega(G) > 2$, then 
\[ \chi(G) \le 200\omega(G) \frac{\Delta \ln \ln \Delta}{\ln \Delta}. \]
\end{itemize}
\end{thm}
}

The first bound is sharp up to a multiplicative factor in a strong sense, since
as shown by Bollobás~\cite[Corollaries~3 and~4]{Bol81} for all integers~$g$
and~$\Delta\ge3$ there exists a graph with maximum degree~$\Delta$, girth at least~$g$
and chromatic number at least~$\frac{\Delta}{2\ln \Delta}$.

\raggedbottom
There remains however a substantial range of degrees not concerned by the bound
for triangle-free graphs given by~Theorem~\ref{thm-molloy}, namely
when~$\Delta$ is smaller than~$\Delta_\eps$, which is larger
than~$20^{2/\varepsilon}$.
Determining the maximum value of~$\chi_f$ among triangle-free graphs of maximum degree~$3$ has been a long
standing open problem, before it was settled~\cite{DSV14}. The authors showed it to be equal to~$\rho(3,4)$,
namely~$14/5$. The same question for larger values of the maximum degree is still open; for graphs of maximum
degree~$4$ that value lies between~$3.25$ and~$3.5$.  To this date, the best known upper bound in
terms of clique number and maximum degree (when those two parameters are not too far apart) for the \emph{fractional} chromatic
number\footnote{\vtop{For the chromatic number, the reader is referred to a
nice theorem of Kostochka~\cite{Kos76}, which for instance implies that every
graph with maximum degree at most~$5$ and girth at least~$35$ has chromatic
number at most~$4$ (Corollary~2 in \emph{loc.\ cit.}).  The general upper bound
on the chromatic number guaranteed by Kostochka's theorem is never less than
the floor of half the maximum degree plus two.}} is due to Molloy and
Reed~\cite[Theorem~21.7, p.~244]{MoRe02}.

\begin{thm}[Molloy and Reed, 2002]\label{thm:Reedsbound}
For every graph~$G$,
\[ \chi_f(G) \le \frac{\omega(G)+\Delta(G)+1}{2}. \]
\end{thm}

If one considers a convex combination of the clique number and the maximum
degree plus one for an upper bound on the (fractional) chromatic number of a graph, then
because the chromatic number of a graph never exceeds its maximum degree plus one, the
aim is to maximise the coefficient in front of the clique number. The convex
combination provided by Theorem~\ref{thm:Reedsbound} (which is conjectured to
hold, after taking the ceiling, also for the chromatic number), is best
possible. Indeed, for every positive integer~$k$ the graph~$G_k \coloneqq C_5
\boxtimes K_k$ is such that~$\omega(G_k)=2k, \Delta(G_k)=3k-1,
\chi_f(G_k)=\frac{5k}{2} = \frac{\omega(G_k)+\Delta(G_k)+1}{2}$.

A local form of \Cref{thm:Reedsbound} exists: it was first devised by McDiarmid
(unpublished) and appears as an exercise in Molloy and Reed's
book~\cite{MoRe02}. A published version is found in the Ph.\@D.\@ thesis of Andrew
King~\cite[Theorem~2.10, p.~12]{Kin09}.  

\begin{thm}[McDiarmid, unpublished]\label{thm-McDiarmid}
Let~$G$ be a graph, and set~$f_G(v) \coloneqq
    \frac{\omega_G(v)+\deg_G(v)+1}{2}$ for every~$v\in V(G)$,
    where~$\omega_G(v)$ is the order of a largest clique in~$G$ containing~$v$.
    Then
    \[ \chi_f(G) \le \max\sst{f_G(v)}{v\in V(G)}. \]
\end{thm}
\noindent
In Subsection~\ref{sub-reed}, we slightly strengthen the local property of
\Cref{thm-McDiarmid} as a way to illustrate the arguments used later on.

\subsection{Our results}

Our first contribution is to establish a (non-explicit) formula for an upper
bound on the fractional chromatic number of triangle-free graphs depending on
their maximum degree~$\Delta$. The upper bound which can be effectively
computed from this formula improves on the one which can be derived from
\Cref{thm-McDiarmid} as soon as~$\Delta \ge 17$.

\begin{table}[!ht]\centering
\begin{tabular}{@{ }ccccccc@{ }}
\toprule
~$\Delta(G)$&\phantom{abc}&$k$&\phantom{abc}&$\lambda$&\phantom{abc}& upper bound on~$\chi_f(G)$\\
\midrule
~$1 \dotso 16$&\phantom{abc} &~$2$&\phantom{abc} &~$\infty$ &\phantom{abc}&$\frac{\Delta(G)+3}{2}$ \\
~$17$ &\phantom{abc}&~$3$ &\phantom{abc}&~$3.41613$ &\phantom{abc} &~$9.91552$  \\
~$18$ &\phantom{abc} &~$3$ &\phantom{abc} &~$3.50195$ &\phantom{abc} &~$10.3075$  \\
~$19$ &\phantom{abc} &~$3$ &\phantom{abc} &~$3.58603$ &\phantom{abc} &~$10.6981$  \\
~$20$ &\phantom{abc} &~$3$ &\phantom{abc} &~$3.66847$ &\phantom{abc} &~$11.0875$  \\
~$50$ &\phantom{abc} &~$4$ &\phantom{abc} &~$2.04455$ &\phantom{abc} &~$ 22.1644$  \\
~$100$ &\phantom{abc} &~$5$ &\phantom{abc} &~$1.48418~$ &\phantom{abc} &~$ 38.0697$  \\
~$200$ &\phantom{abc} &~$6$ &\phantom{abc} &~$1.24061~$ &\phantom{abc} &~$66.151~$  \\
~$500$ &\phantom{abc} &~$8$ &\phantom{abc} &~$0.915598~$ &\phantom{abc} &~$ 139.842$  \\
~$1000$ &\phantom{abc} &~$10$ &\phantom{abc} &~$0.734978~$ &\phantom{abc} &~$ 249.058$  \\
\bottomrule
\end{tabular}
\vspace{6pt}
    \caption{Upper bounds on~$\chi_f(G)$ when~$G$ is triangle-free.}\label{tab:3}
\end{table}

\begin{thm}\label{thm:triangle-free}
For every triangle-free graph~$G$ of maximum degree~$\Delta$, 
    \[ \chi_f(G) \le 1 + \min_{k\in \mathbb{N}} \inf_{\lambda>0} \frac{(1+\lambda)^k+\lambda(1+\lambda)\Delta}{\lambda(1+k\lambda)}.\]
\end{thm}

\Cref{thm:triangle-free} lets us derive the upper bounds for the fractional
chromatic number of triangle-free graphs presented in Table~\ref{tab:3}.
We note that considering the couple~$(k,\lambda) = (2,\infty)$,
Theorem~\ref{thm:triangle-free}  implies the fractional Reed bound of
Theorem~\ref{thm:Reedsbound}.
We also obtain the following upper bound as a corollary.
    
\begin{cor}\label{cor-general} For every triangle-free graph~$G$ of maximum degree~$\Delta\ge 2$,
\[ \chi_f(G) \le 1+\pth{1+\frac{2}{\ln \Delta}} \frac{\Delta}{\ln \Delta - 2\ln \ln \Delta}.\]
\end{cor}
\begin{proof}
The function~$f\colon \Delta \mapsto 1+\pth{1+\frac{2}{\ln \Delta}} \frac{\Delta}{\ln \Delta - 2\ln \ln \Delta}$ is bounded from below by~$4$ when~$\Delta >0$, so we may assume that~$\Delta \ge 4$, using the naive upper bound~$\Delta+1$ for smaller values of~$\Delta$.

It remains to apply \Cref{thm:triangle-free} and consider the
    couple~$(k,\lambda) = (\lfloor \ln \Delta(\ln \Delta - 2\ln \ln
    \Delta)\rfloor, \frac{1}{\ln \Delta})$, by noting that we then have~$k \lambda \le \ln \Delta - 2\ln \ln \Delta \le 1+k\lambda$, and~$(1+\lambda)^k \le e^{k\lambda} \le \Delta/(\ln \Delta)^2$.
\end{proof}

We note that \Cref{cor-general} in particular implies the fractional version of the
triangle-free bound of Theorem~\ref{thm-molloy}. Therefore
Theorem~\ref{thm:triangle-free} yields a smooth transition for the fractional
chromatic number of triangle-free graphs from Reed's bound to Molloy's bound,
as~$\Delta$ increases.

\bigskip

In order to obtain upper bounds smaller than that of \Cref{thm-McDiarmid} for
smaller values of the maximum degree, we need to consider graphs of higher
girth. Our second contribution is to establish good upper bounds for the
fractional chromatic number of graphs of girth~$7$. Moreover, these bounds have
the same local property as those of \Cref{thm-McDiarmid}.

\begin{thm}\label{thm:girth7}
    Let~$f(x) \coloneqq 1 + \min_{k \in \mathbb{N}}  \frac{2x +  2^{k-3}}{k}$.
If~$G$ is a graph of girth at least~$7$, then~$G$ admits a fractional colouring~$c$ such that for every induced subgraph~$H$ of~$G$, the restriction of~$c$
    to~$H$ has weight at most~$f\left(\max\sst{\deg_G(v)}{v\in V(H)}\right)$.
In particular,
\[ \chi_f(G) \le f(\Delta(G)). \]
\end{thm}

\begin{rk}
    In Theorem~\ref{thm:girth7}, if~$x\ge3$ then
    the minimum of the function~$k\to\tfrac{2x+2^{k-3}}{k}$ (over~$\mathbb{N}$)
    is attained when~$k$ is the integer closest to~$4 + \log_2 x -
    \log_2\log_2 x$. So if~$x\ge3$, then~$f(x) = (2\ln 2 + o(1)) 
    x/ \ln x$, which is off by a multiplicative factor~$2\ln 2$ from the asymptotic value for triangle-free graphs which can be derived from \Cref{thm:triangle-free}. The turning point happens when the maximum degree is approximately~$3\cdot 10^6$.
Theorem~\ref{thm:girth7} lets us derive the upper bounds for the fractional chromatic number of graphs of girth at least~$7$ presented in Table~\ref{tab:4}.
\end{rk}

\renewcommand{\arraystretch}{1.4}
\begin{table}[!ht]\centering
\begin{tabular}{@{ }ccccc@{ }}
\toprule
~$\Delta(G)$&\phantom{abc}&optimal~$k$&\phantom{abc}&upper bound on~$\chi_f(G)$\\
\midrule
~$3 \dotso 8$&\phantom{abc} &~$5$&\phantom{abc} &$\frac{2\Delta(G)+9}{5}$ \\
~$8 \dotso 20$&\phantom{abc} &~$6$&\phantom{abc} &$\frac{\Delta(G)+7}{3}$ \\
~$20 \dotso 48$&\phantom{abc} &~$7$&\phantom{abc} &$\frac{2\Delta(G)+23}{7}$ \\
~$48 \dotso 112$&\phantom{abc} &~$8$&\phantom{abc} &$\frac{\Delta(G)}{4}+5$ \\
~$112 \dotso 256$&\phantom{abc} &~$9$&\phantom{abc} &$\frac{2\Delta(G)+73}{9}$ \\
\bottomrule
\end{tabular}
\vspace{6pt}
    \caption{Upper bounds on~$\chi_f(G)$ when~$G$ has girth at least~$7$.}\label{tab:4}
\end{table}
\renewcommand{\arraystretch}{1}

One could wonder to what extent our results extend to the chromatic number. This is the motivation of a follow-up work involving the authors~\cite{DKPS20+}, where the main theorem is a version of \Cref{algo} holding for DP-colourings, which is used in order to derive bounds on~$\chi_{\rm DP}$ for various classes of sparse graphs. However, the bound of that main theorem is looser than that of \Cref{algo} in several ways, which makes it irrelevant for small degree graphs. Finding a generic method allowing to compute relevant upper bounds for the chromatic number of classes of sparse graphs of small maximum degree is an enticing open problem.

\medskip

Finally, we provide improved upper bounds on the Hall ratio of
graphs of maximum degree in~$\{3,4,5\}$ and girth in~$\{6,\dotsc,12\}$. In
particular, these are upper bounds on the fractional chromatic number of
vertex-transitive graphs in these classes.
These upper bounds are obtained \emph{via} a systematic computer-assisted method.

\begin{thm}\label{thm:ratio}
    The values presented in \Cref{table-thm7} are upper bounds
    on~$\rho(d,g)$ for~$d\in\{3,4,5\}$ and~$g\in\{6,\dotsc,12\}$.
\end{thm}

\begin{table}[!ht]
\begin{center}
    \begin{tabular}{@{}cccccccc@{}}
        \toprule
\diagbox{$d$}{$g$} &~$6$ &~$7$ &~$8$ &~$9$ &~$10$ &~$11$ &~$12$ \\
\midrule 
$3$ &~$30/11 \approx 2.727272$ & \color{gray}~$30/11$ &~$2.625224$ &~$2.604167~$ &~$2.557176$ &~$2.539132$ &~$2.510378$  \\
$4$  &~$41/13 \approx 3.153846$ & \color{gray}~$41/13$ &~$3.038497$ &~$3.017382$ &~$3$  \\
$5$ & \color{gray}~$69/19 \approx 3.631579$ &~$3.6$ &~$ 3.5$ \\
\bottomrule
\end{tabular}
\end{center}
\caption{Upper bounds on~$\rho(d,g)$ for~$d\in\{3,4,5\}$
    and~$g\in\{6,\dotsc,12\}$.}\label{table-thm7}
\end{table}

The bounds provided by Theorem~\ref{thm:ratio} when~$d\in\{3,4\}$ and~$g=7$ are the same as those for~$g=6$.
It seems that this could be a general phenomenon.
We therefore offer the following conjecture, implicitly revealing that
we expect our method to produce an upper bound of~$2.5$ on~$\rho(3,13)$.

\begin{conj}
    The values presented in \Cref{table-conj} are upper bounds on~$\rho(d,g)$ for~$d\in\{3,4,5\}$ and~$g\in\{6,8,10,12\}$.
\end{conj}

\begin{table}[!th]
    \begin{center}
        \begin{tabular}{@{}ccccc@{}} \toprule
\diagbox{$d$}{$g$} &~$6$ &~$8$ &~$10$ &~$12$ \\
\midrule 
$3$ &  &~$2.604167~$ &~$2.539132$ &~$2.5$  \\
$4$  & &~$3.017382$ &~$3$  \\
$5$ &~$3.6$ &~$ 3.5$ \\ \bottomrule
\end{tabular}
\end{center}
\caption{Conjectured upper bounds on~$\rho(d,g)$
    for~$d\in\{3,4,5\}$ and~$g\in\{6,8,10,12\}$.}\label{table-conj}
\end{table}

\subsection{Notations}
We introduce some notations before establishing a few technical lemmas, from which we will prove
Theorems~\ref{thm:girth7} and~\ref{thm:ratio}. 
If~$v$ is a vertex of a graph~$G$ and~$r$ a non-negative integer, then~$N^r_G(v)$ is the set
of all vertices of~$G$ at distance exactly~$r$ from~$v$ in~$G$,
while~$N^r_G[v]$ is~$\bigcup_{j=0}^{r}N^j_G(v)$. If~$u$ is also a vertex of~$G$,
we write~$\di_G(u,v)$ for the distance in~$G$ between~$u$ and~$v$.
Further, if~$J$ is a subset of vertices of~$G$, then we write~$N_G(J)$ for the set of
vertices that are not in~$J$ and have a neighbour in~$J$, while~$N_G[J]$ is~$N_G(J)\cup J$.
We will omit the graph
subscript when there is no ambiguity, and sometimes write~$N_X(v)$ instead
of~$N(v)\cap X$, for any subset of vertices~$X\subseteq V(G)$.
The set of all independent sets of~$G$ is~$\cI(G)$, while~$\cIm(G)$ is the set of all maximal independent sets of~$G$
and~$\cIM(G)$ is the set of all maximum independent sets of~$G$.
If~$w$ is a mapping from~$\cI(G)$ to~$\mathbb{R}$ then for every
vertex~$v\in V(G)$ we set
    \[
        w[v]\coloneqq\sum_{\substack{I\in\cI(G)\\v\in I}}w(I).
    \]
Further, if~$\cI$ is a collection of independent sets of~$G$, then
$w(\cI)\coloneqq\sum_{I\in \cI}w(I)$.  If~$I$ is an independent set of a graph~$G$,
a vertex~$v$ is \emph{covered} by~$I$ if~$v$ belongs to~$I$ or has a neighbour
in~$I$. A vertex that is not covered by~$I$ is \emph{uncovered} (by~$I$).
If~$G$ is a graph rooted at a vertex~$v$, then for every positive integer~$d$,
the set of all vertices at distance~$d$ from~$v$ in~$G$ is a \emph{layer}
of~$G$.

\section{Technical lemmas} 
In this section we present the tools needed for the proofs of the main theorems.

\subsection{Greedy fractional colouring algorithm} 
Our results on fractional colouring are obtained using a greedy algorithm analysed in a recent
work involving the first author~\cite{DJKP20}. This algorithm is a
generalisation of an algorithm first described in the book of Molloy and
Reed~\cite[p.~245]{MoRe02} for the uniform distribution over maximum
independent sets. The setting here is, for each induced subgraph~$H$ of the graph
we wish to fractionally colour, a probability distribution over the
independent sets of~$H$. 

\begin{lemma}[Davies \emph{et al.}, 2018]\label{algo}
Let~$G$ be a graph given with parameters~$\alpha_v, \beta_v$ for every vertex~$v\in V(G)$. For every induced subgraph~$H$ of~$G$, let~$\bI_H$ be a random independent set of~$H$ drawn according to a given probability distribution, and assume that
\[ \alpha_v \pr{v\in \bI_H} + \beta_v \esp{\abs{N(v)\cap \bI_H}} \geq 1,\]
for every vertex~$v\in V(H)$. 
Then the \emph{greedy fractional algorithm} defined by \Cref{alg-greedy} produces a fractional colouring~$w$ of~$G$ such that the restriction of~$w$ to any subgraph~$H$ of~$G$ is a fractional colouring of~$H$ of weight at most~$\max_{v\in V(H)} \limits \alpha_v + \beta_v \deg_G(v)$.
In particular,
\[\chi_f(G)\le \max_{v\in V(G)} \limits \alpha_v + \beta_v \deg_G(v).\]
\end{lemma}

\medskip
\begin{algorithm}\caption{The greedy fractional algorithm}\label{alg-greedy}
    \begin{algorithmic}[0]
        \For{$I\in \cI(G)$}
            \State~$w(I)\gets0$
        \EndFor
        \State~$H\gets G$
        \While{$\abs{V(H)}>0$}
            \State~$\displaystyle \val \gets \min\left\{ \min_{v\in V(H)} \frac{1-w[v]}{\pr{v\in \bI_H}}, \min_{v\in V(H)} \Big(\alpha_v+\beta_v\deg_G(v)\Big)-w\big(\cI(G)\big)\right\}$
            \For{$I\in\cI(H)$}
                \State~$w(I)\gets w(I)+\pr{\bI_{H}=I} \val$
            \EndFor
            \State~$H \gets H - \sst{v\in V(H)}{w[v]=1}$
        \EndWhile
    \end{algorithmic}
\end{algorithm}

We note that in Lemma~\ref{algo}, although there is one probability distribution on each induced subgraph, the reals~$\alpha_v$ and~$\beta_v$ associated with each vertex are fixed once and for all, which somewhat ties together the different probability distributions involved.

\subsection{Hard-core model} 
In the setting of \Cref{algo}, we need a probability distribution over the independent sets of a given graph~$H$. For instance, Molloy and Reed used the uniform distribution over the maximum independent sets of~$H$, and obtained the fractional Reed bound as a result (see Theorem~\ref{thm:Reedsbound}). As we will show in \Cref{sub-reed}, this bound is best possible when restricting to the maximum independent sets, even for trees. Therefore, we need to include non-maximum independent sets with non-zero probability in order to hope for improved bounds. Moreover, in order to perform a local analysis of the possible random outcomes, we need our probability distribution to have good relative independence between the random outcomes in a local part of the graph, and the ones outside this part. 

The probability distribution that we are going to use as a setting of \Cref{algo} is the hard-core distribution over the independent sets of a graph, which has the Spatial Markov Property. Given a family~$\cI$ of independent sets of a graph~$H$, and a positive real~$\lambda$, a random independent set~$\bI$ drawn according to the hard-core distribution at fugacity~$\lambda$ over~$\cI$ is such that
\[ \pr{\bI=I} = \frac{\lambda^{\abs{I}}}{Z_{\cI}(\lambda)},\]
for every~$I\in \cI$, where~$Z_\cI(\lambda)=\sum_{J\in \cI}\limits \lambda^{\abs{J}}$ is the \emph{partition function} associated with~$\bI$.

Along this work, we consider two possible families~$\cI$ of independent sets
of~$H$, the first one being the whole set~$\cI(H)$ of independent sets of~$H$. Note
that when~$\cI=\cI(H)$, and~$\lambda\to \infty$, the hard-core distribution
converges towards the uniform distribution over the maximum independent sets of~$H$.

\begin{lemma}[Spatial Markov Property]\label{lem:markov}
Given a graph~$H$, and a real~$\lambda>0$, let~$\bI$ be drawn according to the
hard-core distribution at fugacity~$\lambda$ over the independent sets~$\cI(H)$
of~$H$. Let~$X\subseteq V(H)$ be any given subset of vertices, and~$J$ any
possible outcome of~$\bI\setminus X$. Then, conditioned on the fact that~$\bI
\setminus X = J$, the random independent set~$\bI\cap X$ follows the hard-core
distribution at fugacity~$\lambda$ over the independent sets of~$H[X\setminus
N(J)]$.
\end{lemma}

The proof of this result is standard and follows from a simple consideration of the marginal
probabilities. It remains valid when we fix~$\lambda=\infty$, i.e. the uniform distribution over the maximum independent sets of any graph~$H$ has the Spatial Markov Property.
Things are more complicated with our second choice for~$\cI$, that is the
set~$\cIm(H)$ of maximal independent sets of~$H$. Indeed, in this setting, one has to make sure
that the local outcome of the independent set is compatible with the fact that the global outcome
of the independent set is maximal, i.e. there remains no uncovered vertices in~$H$. This
adds a new level of dependency, and we need the extra assumption~\eqref{eq:markov} introduced below in Lemma~\ref{lem:markov-max}, to be
able to handle it. For two disjoint subsets of vertices~$X$ and~$U$ of a graph~$G$, we
define~$P^2_X(U)$ to be the set of vertices~$x\in X$ such that there exists a path~$ux'x$ of
length~$2$ with~$u\in U$ and~$x'\in X$. In symbols, $P^2_X(U)=N\big(N(U)\cap X\big) \cap X$.

\begin{lemma}[Spatial Markov Property for maximal independent sets]\label{lem:markov-max}
Given a graph~$H$, and a real~$\lambda>0$, let~$\bI$ be drawn according to the
    hard-core distribution at fugacity~$\lambda$ over the maximal independent
    sets~$\cIm(H)$ of~$H$. Let~$X\subseteq V(H)$ be any given subset of
    vertices,~$J$ any possible outcome of~$\bI\setminus X$, and~$U\coloneqq
    (V(H)\setminus X) \setminus N[J]$ the set of vertices outside of~$X$
    that are uncovered (by~$J$). Moreover, we assume that
\begin{equation}\label{eq:markov}
    \abs{N(v) \cap X} \le 1\quad\text{for any vertex~$v\in V(H)\setminus X$.} \tag{$\star$}
\end{equation}
Then, conditioned on the fact that~$\bI \setminus X
    = J$, the random independent set~$\bI\cap X$ follows the hard-core distribution
    at fugacity~$\lambda$ over the maximal independent sets of
    \[H[X\setminus (N(J)\cup P^2_X(U)].\]
\end{lemma}

\begin{proof}
    Set~$W \coloneqq X \setminus (N(J)\cup P^2_X(U))$.
First let~$I_X$ be any possible realisation of~$\bI \cap X$, conditioned on the fact
    that~$\bI\setminus X = J$.
    We prove that~$I_X \in \cIm(H[W])$. To this end, we begin by showing that~$N_X(U) \subseteq I_X$. By the definitions of~$J$ and~$U$,
    every vertex in~$U$ must be adjacent to a vertex in~$X$, and hence by~\eqref{eq:markov} for each~$u\in U$
    there exists a unique vertex~$v_u$ in~$X$ that is adjacent to~$u$. It follows that~$N_X(U)$
    is contained in~$I_X$. This in particular implies that no vertex in~$P^2_X(U)$
    can belong to~$I_X$, and hence~$I_X\subseteq W$. We note for later that we just
    established that~$N_X(U)$ is a set of isolated vertices of~$H[W]$ (that is, these
    vertices belong to~$W$ and have no neighbour in~$H[W]$).

    Next we observe that~$I_X$ is maximal in~$H[W]$. Indeed, let~$w\in W\setminus I_X$. Because~$I_X\cup J$
    is a maximal independent set of~$H$, there exists~$v\in I_X\cup J$ that is adjacent to~$w$ in~$H$.
    Since~$W\subseteq X\setminus N(J)$ by definition, we deduce that~$v\in I_X$ and hence~$I_X$ is maximal
    in~$H[W]$.

Second, given any set~$I_X \in \cIm(H[W])$, the set~$I_X\cup J$ is a
valid realisation of~$\bI$. Indeed,~$I_X$ and~$J$ are independent sets, and so is their union
    as~$I_X \cap N(J) = \varnothing$. To prove that~$I_X\cup J$ is maximal in~$H$, it
    suffices to show that every vertex~$x$ in~$U\cup (X\setminus W)$ has a neighbour
    in~$I_X$. As reported earlier,~$N_X(U)$ is contained in~$W$ and forms a set
    of isolated vertices in~$H[W]$. Therefore,~$N_X(U)$ is contained in every maximal independent
    of~$H[W]$, and hence in~$I_X$. Since every vertex in~$U$ has a neighbour in~$X$,
    it therefore only remains to deal with the case where~$x\in X\setminus W$.
    Then~$x\in N(J)\cup P^2_X(U)$, and hence~$x$ has a neighbour in~$J\cup N_X(U)$, which
    is contained in~$I_X\cup J$.

In conclusion, the set of realisations of~$\bI \cap X$ is exactly~$\cIm(H[W])$,
    and each such realisation~$I_X$ has a probability proportional to
~$\lambda^{\abs{I_X}+\abs{J}}$, and hence proportional to
~$\lambda^{\abs{I_X}}$ since~$J$ is fixed. This finishes the proof.
\end{proof}

%

\subsection{Independence ratio} 
We state two lemmas which can be proved in similar ways. We only present the
proof of the second one, the argument for the first one being very close but
a little simpler.

\begin{lemma}\label{ratio-vertex}
    Let~$r$ be a positive integer and~$G$ be a~$d$-regular graph. Let~$\alpha_0,\dotsc,\alpha_r$ be real
    numbers such that~$\sum_{i=1}^{r}\alpha_i(d-1)^{i-1}\ge0$.  Assume that there exists a probability
    distribution~$p$ on~$\cI_{\max}(G)$ such that
\begin{equation}
\forall v\in V(G),\quad \sum_{i=0}^r \alpha_i\esp{\mathbf{X}_i(v)} \geq 1,
\end{equation}
where~$\mathbf{X}_i(v)$ is the random variable counting the number of paths of
    length~$i$ between~$v$ and a vertex belonging to a random independent set~$\bI$
    chosen following~$p$.  Then
\begin{equation}
    \frac{\abs{V(G)}}{\alpha(G)} \le \alpha_0 + \sum_{i=1}^{r} \alpha_i d(d-1)^{i-1}.
\end{equation}
\end{lemma}

\begin{lemma}\label{ratio-edge}
    Let~$r$ be a positive integer and~$G$ be a~$d$-regular graph. Let~$\alpha_0,\dotsc,\alpha_r$ be real
    numbers such that~$\sum_{i=0}^{r}\alpha_i(d-1)^{i}\ge0$.  Assume that there exists a probability
    distribution~$p$ on~$\cI_{\max}(G)$ such that 
\begin{equation}
\label{edge-inequality}
\forall e\in E(G),\quad \sum_{i=0}^r \alpha_i\esp{\mathbf{X}_i(e)} \geq 1,
\end{equation}
where~$\mathbf{X}_i(e)$ is the random variable counting the number of paths of length~$i+1$ starting with~$e$ and ending at a vertex belonging to a random independent set~$\bI$ chosen following~$p$.
Then
\begin{equation}
    \frac{\abs{V(G)}}{\alpha(G)} \le \sum_{i=0}^{r} 2\alpha_i (d-1)^i.
\end{equation}
\end{lemma}

\begin{proof}
    Given an integer~$i\in\{0,\dotsc,r\}$ and an edge~$e$ of~$G$, the contribution of an arbitrary vertex~$v\in\bI$
    to~$\mathbf{X}_i(e)$ is the number of paths of length~$i+1$ starting at~$v$
    and ending with~$e$.  It follows that the total contribution of any
    vertex~$v\in\bI$ to~$\sum_{e\in E(G)}\mathbf{X}_i(e)$ is the number of
    paths of~$G$ with length~$i+1$ that start at~$v$, which is at most~$d(d-1)^i$
    since~$G$ is a~$d$-regular graph.  Consequently,
    \[ \esp{\sum_{e\in E(G)} \mathbf{X}_i(e)} \le \sum_{v\in V(G)} \pr{v\in \bI} d(d-1)^i. \]
We now sum~\eqref{edge-inequality} over all edges of~$G$.

\begin{align*}
    \sum_{e\in E(G)} \sum_{i=0}^r \alpha_i \esp{\mathbf{X}_i(e)} &\ge  \abs{E(G)} = \frac{d\cdot\abs{V(G)}}{2} \\
    \sum_{i=0}^r  \alpha_i \sum_{e \in E(G)} \esp{\mathbf{X}_i(e)} &\ge  \frac{d\cdot\abs{V(G)}}{2} \\
    \sum_{i=0}^r  \alpha_i \sum_{v\in V(G)} \pr{v\in \bI} d(d-1)^i &\ge  \frac{d\cdot\abs{V(G)}}{2} \\
    \sum_{i=0}^r 2 \alpha_i \esp{\abs{\bI}} (d-1)^i &\ge  \abs{V(G)} \\
    \sum_{i=0}^{r} 2\alpha_i (d-1)^i &\ge  \frac{\abs{V(G)}}{\alpha(G)}.
\qedhere
\end{align*}
\end{proof}

The next lemma allows us to generalise Lemmas~\ref{ratio-vertex} and~\ref{ratio-edge} to
non-regular graphs. To this end, we use a standard argument coupled with the existence of
specific vertex-transitive type-$1$ regular graphs with any given degree and girth. These
are provided by a construction of Exoo and Jajcay~\cite{ExJa08} in the proof of their
Theorem~19, which is a direct generalisation of a construction for cubic graphs designed by
Biggs~\cite[Theorem~6.2]{Big98}. We slightly reformulate their theorem, the mentioned
edge-colouring and transitivity property following simply from the fact that the graph constructed is
a Cayley graph obtained from a generating set consisting only of involutions.
Given a graph~$G$ endowed with an edge-colouring~$c$, an automorphism~$f$
of~$G$ is \emph{$c$-preserving} if~$c(\{f(u),f(v)\}) = c(u,v)$ for each edge~$\{u,v\}$
of~$G$. The graph~$G$ is \emph{$c$-transitive} if for every pair~$(u,v)$ of
vertices of~$G$ there exists a~$c$-preserving automorphism~$f$ of~$G$ such that~$f(u)=v$.

\begin{thm}[Exoo \& Jajcay, 2013]\label{thm-exja}
For every integers~$d$ and~$g$ both at least~$3$, there exists a
~$d$-regular graph~$H$ with girth at least~$g$ along with a proper edge-colouring~$c$
    using~$d$ colours such that~$H$ is~$c$-transitive.
\end{thm}

\begin{lemma}\label{regular}
From any graph~$G$ of maximum degree~$d$ and girth~$g$, we can construct a
$d$-regular graph~$\varphi(G)$ of girth~$g$ whose vertex set can be
partitioned into induced copies of~$G$, and such that any vertex~$v\in G$ can
be sent to any of its copies through an automorphism.
\end{lemma}

\begin{proof}
Set~$k \coloneqq \sum_{v\in G} (d-\deg(v))$. Let~$G'$ be the supergraph of~$G$ obtained by
    adding~$k$ vertices~$v'_1,\dotsc,v'_k$ each of degree~$1$, such that all
    other vertices have degree~$d$. We let~$e'_i$ be the edge of~$G'$ incident to~$v'_i$,
    for each~$i\in\{1,\dotsc,k\}$. By Theorem~\ref{thm-exja}, there exists a
~$k$-regular graph~$H$ of girth at least~$g$ together with a proper edge-colouring~$c$
    using~$k$ colours, such that~$H$ is~$c$-transitive.  Let~$n(H)$ be the number of
    vertices of~$H$ and write~$V(H)=\{1,\dotsc,n(H)\}$.

    We construct~$\varphi(G)$ by starting from the disjoint union of~$n(H)$
    copies~$G_1,\dotsc,G_{n(H)}$ of~$G$.  For each edge~$e=\{i,j\}\in E(H)$, letting~$u_e$
    be the vertex of~$G$ incident to the edge~$e'_{c(e)}$ in~$G'$, we add an edge between
    the copy of~$u_e$ in~$G_i$ and that in~$G_j$.

Any cycle in~$\varphi(G)$ either is a cycle in~$G$, and hence has length at least~$g$, or
    contains all the edges of a cycle in~$H$, and hence has length at least~$g$. It follows
    that~$\varphi(G)$ has girth~$g$.

    The last statement follows directly from the fact that~$H$ is~$c$-transitive.
\end{proof}

\begin{cor}
    Let~$d$ and~$g$ be integers greater than two. If there exists a constant~$B=B(d,g)$
    such that every~$d$-regular graph~$H$ with girth~$g$ has independence ratio
    at least~$B$, then every graph~$G$ with maximum degree~$d$ and
    girth~$g$ also has independence ratio at least~$B$.
In particular, if \Cref{ratio-vertex} or \Cref{ratio-edge} can be applied to
    the class of~$d$-regular graphs of girth~$g$, then the conclusion also holds
    for the class of graphs with maximum degree~$d$ and
    girth~$g$, that is, for~$\rho(d,g)$.
\end{cor}

\begin{proof}
    Let~$G$ be a graph with maximum degree~$d$ and girth~$g$ on~$n$ vertices. Let~$\varphi(G)$
    be the graph provided by \Cref{regular}. 
    In particular,~$\abs{V(\varphi(G))}=kn$ where~$k$ is the number of induced copies of~$G$ partitioning~$V(\varphi(G))$.
    By assumptions,~$\varphi(G)$ contains an independent set~$I$ of order at least~$B\cdot kn$. Letting~$I_i$ be the
    set of vertices of the~$i$-th copy of~$G$ contained in~$I$, by the pigeon-hole principle there exists~$i\in\{1,\dotsc,k\}$
    such that~$\abs{I_i}\ge B\cdot n$, and hence~$G$ has independence ratio at least~$B$.
\end{proof}

\section{Fractional colourings} 

\subsection{A local version of Reed's bound}\label{sub-reed} 
For the sake of illustration, we begin by showing how \Cref{algo} can be used
to prove \Cref{thm:Reedsbound}.  We actually establish a slight strengthening of
\Cref{thm-McDiarmid}, the local form of \Cref{thm:Reedsbound}.  The argument
relies on the relation~\eqref{reedinequality} below~\cite[Lemma~2.11]{Kin09},
which is a local version of the relation~(21.10) appearing in Molloy and Reed's
book~\cite{MoRe02}.  The short argument, however, stays the same and we provide
it here only for explanatory purposes, since it is the inspiration for the
argument used in the proof of \Cref{thm:girth7}.

\begin{prop}\label{prop-McDiarmidlike}
Let~$G$ be a graph, and set~$f_G(v) \coloneqq
    \frac{\omega_G(v)+\deg_G(v)+1}{2}$ for every~$v\in V(G)$,
    where~$\omega_G(v)$ is the order of a largest clique in~$G$ containing~$v$.
    Then~$G$ admits a fractional colouring~$c$ such that the restriction of~$c$
    to any induced subgraph~$H$ of~$G$ has weight at most~$\max_{v\in V(H)}\limits f_G(v)$.
    In particular,
    \[ \chi_f(G) \le \max\sst{f_G(v)}{v\in V(G)}. \]
\end{prop}

\begin{proof}
    We demonstrate the statement by applying \Cref{algo}. To this end, we use
    the uniform distribution on maximum independent sets, which corresponds to the
    hard-core distribution at fugacity~$\lambda=\infty$.
\begin{as}\label{reedinequality}
For every induced subgraph~$H$ of~$G$, let~$\bI_H$ be a maximum independent set of~$H$,
    drawn uniformly at random. Then for every vertex~$v\in V(H)$,
\[\frac{\omega(v)+1}{2} \pr{v\in \bI_H} + \frac{1}{2} \esp{\abs{N(v)\cap \bI_H}} \ge 1.\]
\end{as}
\noindent
The conclusion follows by applying \Cref{algo} with
~$\alpha_v=\frac{\omega(v)+1}{2}$ and~$\beta_v=\frac12$ for every~$v\in V(G)$.

\smallskip
    It remains to establish~\eqref{reedinequality}. We let~$J$ be any possible outcome
    of~$\bI_H \setminus N[v]$, and~$W = N[v]\setminus N(J)$. We condition on the random
    event~$E_J$ that~$\bI_H \setminus N[v]=J$, and the Spatial Markov Property of the
    uniform distribution over the maximum independent sets of~$H$ ensures that~$\bI_H \cap W$ is
    a uniform random maximum independent set of~$H[W]$. There are two cases.

\begin{enumerate}[label=(\roman*)]
    \item If~$W$ is a clique of size~$k \le \omega(v)$, then exactly one vertex from~$W$ belongs to~$\bI_H$, and every vertex in~$W$ has equal probability~$1/k$ to be in~$\bI_H$. So, in this case,
        \[ \frac{\omega(v)+1}{2} \pr{v\in \bI_H \mid E_J} + \frac{1}{2}
        \esp{\abs{N(v)\cap \bI_H} \mid E_J} = \frac{\omega(v)+1}{2k} +
        \frac{k-1}{2k}  \ge  1. \]
\item If~$W$ is not a clique, then~$\abs{W\setminus \{v\} \cap \bI_H} \ge 2$ and~$v\notin \bI_H$, since~$\bI_H$ is a maximum independent set.
        So, in this case,
        \[\frac{\omega(v)+1}{2} \pr{v\in \bI_H \mid E_J} + \frac{1}{2} \esp{\abs{N(v)\cap \bI_H} \mid E_J} \ge \frac{1}{2} \times 2 = 1.
        \]
\end{enumerate}
    The validity of~\eqref{reedinequality} follows by summing over all possible
realisations~$J$ of~$\bI_H \setminus N[v]$. 
\end{proof}

We finish by noting that the bound provided by \Cref{thm-McDiarmid} is best
possible over the class of unicyclic triangle-free graphs if one uses the
fractional greedy colouring of \Cref{algo} together with any probability
distribution on the \emph{maximum} independent sets of the graph.

\begin{lemma}\label{lem-best}
    If the probability distribution used in \Cref{algo} gives positive
    probability only to maximum independent sets, then the greedy fractional
    colouring algorithm can return a fractional colouring of weight up
    to~$\frac{d+3}{2}$ in general for graphs of degree~$d$, should they be
    acyclic when~$d$ is odd, or have a unique cycle (of length~$5$) when~$d$ is
    even.
\end{lemma}

\begin{proof}
We prove the statement by induction on the positive integer~$d$.
\begin{itemize}
\item If~$d=1$, then let~$G_1$ consist only of an edge. The algorithm returns a fractional colouring of~$G_1$ of weight~$2$.
\item If~$d=2$, then let~$G_2$ be the cycle of length~$5$. The algorithm returns a fractional colouring of~$G_2$ of weight~$\frac{5}{2}$.
\item If~$d>2$, then let~$G_d$ be obtained from~$G_{d-2}$ by adding two neighbours of
    degree~$1$ to every vertex. This creates no new cycles, so~$G_d$ is acyclic
        when~$d$ is odd, and contains a unique cycle, which is of length~$5$,
        when~$d$ is even.

For every~$d\ge3$, the graph~$G_d$ contains a unique maximum independent set,
        namely~$I_0 \coloneqq V(G_d)\setminus V(G_{d-2})$. After the first
        step of the algorithm applied to~$G_d$, all the vertices in~$I_0$ have
        weight~$1$, and we are left with the graph~$G_{d-2}$ where every vertex
        has weight~$0$. By the induction hypothesis, the total weight of the
        fractional colouring returned by the algorithm is therefore~$1 +
        \frac{(d-2)+3}{2} = \frac{d+3}{2}$.
        \qedhere
\end{itemize}
\end{proof}

\subsection{Triangle-free graphs} 
Using a similar approach, it is possible to obtain improved bounds for the
fractional chromatic number of a given triangle-free graph~$G$, as stated in Theorem~\ref{thm:triangle-free}.

\begin{proof}[Proof of Theorem~\ref{thm:triangle-free}]

Let~$\lambda>0$ be any positive real.
For every induced subgraph~$H$ of~$G$, we let~$\bI_H$ be a random independent set drawn according to the hard-core
distribution at fugacity~$\lambda$ over the set~$\cI(H)$ of all independent sets of~$H$. We first assert the following.

\begin{as}\label{triangle-free-occupancy} For every vertex~$v\in V(H)$, and every integer~$k \ge 1$,
\[\pth{1+\frac{(1+\lambda)^k}{\lambda(1+k\lambda)}} \pr{v\in \bI_H} + \frac{1+\lambda}{1+k\lambda}\esp{\abs{N(v)\cap \bI_H}} \ge 1.\]
\end{as}

\noindent
Note that when~$\omega=2$, we deduce~\eqref{reedinequality}
from~\eqref{triangle-free-occupancy} by taking~$k=2$ and
letting~$\lambda$ go to infinity.

The result follows from \eqref{triangle-free-occupancy} by applying \Cref{algo} with~$\alpha_v =  1+\frac{(1+\lambda)^k}{\lambda(1+k\lambda)}$ and~$\beta_v = \frac{1+\lambda}{1+k\lambda}$ for every~$v\in V(G)$. There remains to prove \eqref{triangle-free-occupancy}.

    We let~$J$ be any possible realisation of~$\bI_H \setminus N[v]$. By the
    Spatial Markov Property of the hard-core distribution, if we condition on
    the event~$E_J$ that~$\bI_H \setminus N[v] = J$ and write~$W \coloneqq
    N[v]\setminus N(J)$, then~$\bI_H \cap N[v]$ follows the hard-core
    distribution at fugacity~$\lambda$ over~$\cI(H[W])$. Since~$G$ (and
    therefore also~$H$) is triangle-free,~$H[W]$ is a star~$K_{1,d}$, for
    some integer~$d\in\{0,\dotsc,\Delta(G)\}$. An analysis of the hard-core
    distribution over the independent sets of a star yields that
    \begin{enumerate}[label=(\roman*)]
\item~$\pr{v\in \bI_H \mid E_J} = \dfrac{\lambda}{\lambda+(1+\lambda)^d}$, and
\item~$\esp{\abs{N(v)\cap \bI_H}\mid E_J} = \dfrac{d\lambda(1+\lambda)^{d-1}}{\lambda+(1+\lambda)^d}$,\label{ite-itr}
    \end{enumerate}
    where~\ref{ite-itr} uses that~$\sum_{i=0}^d i\binom{d}{i}\lambda^i = d\lambda{(1+\lambda)}^{d-1}$.

For some positive real numbers~$\alpha$ and~$\beta$, we let 
\[g(x) \coloneqq  \alpha \dfrac{\lambda}{\lambda+(1+\lambda)^x} + \beta \dfrac{x\lambda(1+\lambda)^{x-1}}{\lambda+(1+\lambda)^x}.\]
We observe that~$g$ is a convex function, and therefore its minimum over the
    (non-negative) reals is reached at its unique critical point~$x^*$ such
    that~$g'(x^*)=0$ (if it exists). Moreover, if there are numbers~$y$ and~$z>y$ such that
~$g(y)=g(z)$, then Rolle's theorem ensures that~$x^* \in (y,z)$, and~$g(x) \ge g(y)$ for every~$x\notin(y,z)$.

Let~$k$ be a positive integer. We now fix
\[ \alpha \coloneqq 1+\frac{(1+\lambda)^k}{\lambda(1+k\lambda)} \quad \et \quad \beta = \frac{1+\lambda}{1+k\lambda}.\]
One can easily check that with these values for~$\alpha$ and~$\beta$, it holds that 
\[ g(k-1)=1=g(k).\]
We conclude that~$g(d) \ge 1$ for every non-negative integer~$d$, which means that
\begin{equation}\label{eq:triangle-free}
\pth{1+\frac{(1+\lambda)^k}{\lambda(1+k\lambda)}} \pr{v\in \bI_H \mid E_J} + \frac{1+\lambda}{1+k\lambda}\esp{\abs{N(v)\cap \bI_H} \mid E_J} \ge 1,
\end{equation}
for any possible realisation~$J$ of~$\bI_H\setminus N[v]$. The conclusion
follows again by taking the convex combination of~\eqref{eq:triangle-free} over all
possible values of~$J$. 
\end{proof}

\subsection{A stronger bound for graphs of girth 7} 

\begin{figure}[t] 
\centering
\begin{tikzpicture}[x=.5cm]

\node[gertex,label=left:\(v\)] (v) {};
\node[wertex,right=1.5cm of v,label=above:\(c\)] (c) {};
\node[wertex,above of=c,label=above:\(b\)] (b) {};
\node[gertex,above of=b,label=above:\(a\)] (a) {};
\node[gertex,below of=c,label=above:\(d\)] (d) {};
\node[gertex,below of=d,label=above right:\(e\)] (e) {};
\node[gertex,left=1cm of e,label=right:\(f\)] (f) {};
\path (a)++(2.8,0) node[wertex,label=above:\(a_1\)] (aa) {};
\path (b)++(2.8,1.5) node[wertex,label=above:\(b_1\)] (bb) {};
\path (b)++(2.8,0) node[vertex,label=above:\(b_2\)] (bbb) {};
\path (b)++(2.8,-1.5) node[vertex,label=above:\(b_3\)] (bbbb) {};
\path (c)++(2.8,0) node[vertex,label=above:\(c_1\)] (cc) {};
\path (c)++(2.8,-1.5) node[vertex,label=above:\(c_2\)] (ccc) {};
\path (d)++(2.8,0) node[gertex,label=above:\(d_1\)] (dd) {};
\path (d)++(2.8,-1.5) node[gertex,label=above:\(d_2\)] (ddd) {};
\path (e)++(2.8,0) node[gertex,label=above:\(e_1\)] (ee) {};
\draw[thick] (bb)--(b)--(v)--(a)--(aa);
\draw[thick] (cc)--(c)--(v)--(d);
\draw[thick] (f)--(v)--(e)--(ee);
\draw[thick] (bbb)--(b)--(bbbb);
\draw[thick] (c)--(ccc);
\draw[thick] (dd)--(d)--(ddd);
\node[wertex,xshift=11cm,yshift=-2cm] (v1) {};
\node[wertex,above= of v1] (v2) {};
\path (v1)--(v2) node[wertex,midway] (vv) {};
\draw[thick] (vv)--(v2);
\node[color=blue!80!white, below = of v1] {\large $N(J)\setminus N^2[v]$};
\node[color=blue!80!white, below = 1.8 of v1] {\large $=V(H)\setminus(N^2[v]\cup U\cup J)$};
\node[vertex,right= of aa] (v3) {};
\node[vertex,below= of v3] (v5) {};
\node[vertex,right= of v3] (v6) {};
\draw[thick] (vv)--(v5)--(v1);
\draw[thick] (v6)--(v2);
\draw[thick] (v3)--(aa);
\draw[thick] (v5)--(bb);
\node[right = of v6, color=green!50!black] {\Large $J$};
\node[wertex,right= of cc] (j1) {};
\node[wertex,below= 1 of j1] (j2) {};
\node[wertex,below= 1 of j2] (j3) {};
\draw[thick] (j2)--(j3);
\node[wertex,below= 1 of j3] (j4) {};
\node[wertex,below= 1 of j4] (j5) {};
\path (j1.north)++(.1,.35) node (t){};
\draw[dashed] (v5.south)++(-.25,-.5)--(t);
\draw[thick] (j1)--(bbb);
\draw[thick] (j2)--(bbbb);
\draw[thick] (j3)--(cc);
\draw[thick] (j4)--(cc);
\draw[thick] (j5)--(cc);
\begin{pgfonlayer}{background}
\draw[edge,color=blue!80!white] (v1) -- (v2);
\begin{scope}[transparency group,opacity=.5]
    \draw[edge,opacity=1,color=green!50!black] (v3)++(0,.003) -- (v5) -- (v6) -- (v3)++(0,.003);
\fill[edge,opacity=1,color=green!50!black] (v3.north) -- (v5.center) -- (v6.center) -- (v3.north);
\end{scope}
\draw[edge,color=purple] (j1) -- (j5);
\end{pgfonlayer}

\node[right = of j3,color=purple] {\large $U$};
    \node[below left = 6 and .25 of v] (s){};
\draw [decorate,decoration={brace,amplitude=10pt,mirror,raise=4pt}]
    (s) --++ (8.25,0) node [black,midway,yshift=-1cm] {\large $N^2[v]$};


\end{tikzpicture}
\caption{A schematic view of the situation for the proof of~(C). No vertex in~$J$
has a neighbour in~$U$. 
Each vertex in~$U$ has a unique neighbour in~$N^2(v)$, which belongs to~$\bI_H$ with probability~$1$ since~$U$ must be covered. 
All vertices that belong to~$\bI_H$ with probability~$1$ are coloured black, and every neighbour of a black vertex belongs to~$\bI_H$ with probability~$0$; such vertices are coloured white. The white vertices in~$N^2[v]$ are either contained in~$N(J)$, or in~$P^2_{N^2[v]}(U)=\{b,c\}$.
The grey vertices are those whose presence in~$\bI_H$ is not determined; together they form the set~$W=\{v\}\cup W_1\cup W_2$,
with~$W_1=\{a,d,e,f\}$ and~$W_2=\{d_1,d_2,e_1\}$.
    We have~$W_{1,0}=\{a,f\}$,~$W_{1,1}=\{e\}$,~$W_{1,2}=\{d\}$ and~$W_{1,j}=\varnothing$ if~$j\ge3$. In particular~$x_{1,0}=2$ and~$x_{1,j}=1$ if~$j\in\{1,2\}$. Note that~$X_0=\{v,a,f\}$ while~$X_1=\cup_{j\ge1}W_{1,j}=\{d,e\}$. Note also that~$H[X_0]$ is a star centered at~$v$, and all connected components of~$H[X_1\cup W_2]$ are stars centered in~$X_1$.}\label{fig-C}
\end{figure}

Let~$G$ be a graph of girth (at least)~$7$ and~$H$ an induced subgraph of~$G$.
We wish to apply \Cref{algo} with an independent set~$\bI_H$ drawn according to the
hard-core distribution at fugacity~$\lambda$ over the set~$\cIm(H)$ of all
\emph{maximal} independent sets of~$H$, for the specific value~$\lambda=4$.
We begin by establishing the following assertion.

\begin{as}\label{as-A}
For every given vertex~$v\in V(H)$ and every integer~$k\ge 4$, 
\[\frac{2^{k-3}+k}{k}\pr{v\in \bI_H} + \frac{2}{k}\esp{\abs{N(v)\cap \bI_H}} \geq 1.\]
\end{as}

\begin{proof}
    Figure~\ref{fig-C} provides a schematic illustration of the following.
Let~$J$ be any possible realisation of~$\bI_H \setminus N^2[v]$. We are going to condition on the random
    event~$E_J$ that~$\bI_H \setminus N^2[v]=J$.  Let~$U \coloneqq (V(H)\setminus N^2[v])\setminus N[J]$ be
    the set of vertices at distance more than~$2$ from~$v$ that are uncovered (by~$J$), and set~$W'\coloneqq
    N^2[v]\setminus(N(J)\cup P^2_{N^2[v]}(U))$ (in that scenario we have $P^2_{N^2[v]}(U)=N^2(U)\cap N(v)$). It follows from the definitions that~$v$ belongs to~$W'$.
    Because the girth of~$G$ is greater than~$6$, no vertex outside of~$N^2[v]$ has more than one neighbour
    in~$N^2[v]$. So \Cref{lem:markov-max} ensures that~$\bI_H \cap W'$ follows the hard-core distribution at
    fugacity~$\lambda$ over the maximal independent sets of~$H[W']$. As the hard-core distribution behaves
    independently on each connected component, it suffices to work with the connected component of~$H[W']$
    that contains~$v$; letting~$W$ be the vertices in this connected component, we simply ignore the vertices in~$W'\setminus W$.
    (Actually, the vertices in~$W'\setminus W$ are isolated in~$H[W']$ and hence belongs to~$\mathbf{I}_H$ with probability~$1$.)

    We let~$W_i$ be the set of vertices in~$W$ at distance~$i$ from~$v$ in~$H[W]$, for~$i\in \{0,1,2\}$,
    and~$W_{1,j}$ be the subset of vertices of~$W_1$ with~$j$ neighbours in~$W_2$. We set~$x_j \coloneqq
    \abs{W_{1,j}}$. Finally, we let~$X_0 \coloneqq \{v\}\cup W_{1,0}$ and~$X_1 \coloneqq W_1 \setminus W_{1,0}$,
    and for every~$i\in \{0,1\}$ we let~$\bI_i$ be the random outcome of~$\bI_H \cap X_i$ under the
    condition~$E_J$.

\medskip
    We first assume that~$x_0\ge 1$. If~$v\notin \bI_0$, then the maximality of~$\bI_H$ ensures that~$\bI_H
    \setminus (X_1\cup W_2)$ equals~$J\cup W_{1,0}$, which we call~$J_0$. Since~$v$ is covered by~$J_0$, \Cref{lem:markov-max} implies
    that, under the additional condition that $v\notin \bI_0$, the random set~$\bI_1$ follows the hard-core
    distribution at fugacity~${\lambda=4}$ over the maximal independent sets of~$H[X_1\cup W_2]$. Moreover,~$\bI_1$
    behaves independently on each connected component of~$H[X_1\cup W_2]$, which are all stars with center in~$X_1$. So for every
    integer~$j\in\{1,\dotsc,d-1\}$ and every vertex~$u\in W_{1,j}$, we have
\[
\pr{u \in \bI_1 \mid v\notin \bI_0} = \frac{1}{1+\lambda^{j-1}}.
\]
It follows that
    \[
\pr{\bI_1 =\varnothing \mid v\notin \bI_0} =\prod_{j=1}^{d-1} \pth{\frac{\lambda^{j-1}}{1+\lambda^{j-1}}}^{x_j},
    \]
and we define~$\zeta$ to be this value.
Let~$p\coloneqq \pr{v\in \bI_0}$;
we have
\begin{align*}
\pr{\bI_1 = \varnothing} &= \pr{\bI_1 = \varnothing \mid v\in \bI_0} \pr{v\in \bI_0} + \pr{\bI_1 = \varnothing \mid v\notin \bI_0} \pr{v\notin \bI_0}
= p +(1-p)\zeta.
\end{align*}
If $\bI_1=\varnothing$, then the maximality of~$\bI_H$ ensures that~$\bI_H \setminus X_0$ equals~$J \cup W_2$, which we call~$J_1$.
    Since every vertex of~$X_1$ is covered by~$J_1$, \Cref{lem:markov-max} implies that, under the
    additional condition that~$\bI_1=\varnothing$, the random set~$\bI_0$ follows the hard-core distribution
    at fugacity~$\lambda$ over the maximal independent sets of~$H[X_0]$. Hence, since~$v$ cannot belong to~$\mathbf{I}_H$
    as soon as~$\mathbf{I}_1\neq\varnothing$,
\begin{align*}
p = \pr{v\in \bI_0} &= \pr{v\in \bI_0 \mid \bI_1=\varnothing}\pr{\bI_1=\varnothing}
= \frac{1}{1+\lambda^{x_0-1}}\Big(p + (1-p)\zeta\Big),\\
    \intertext{so}
p \pth{1+\lambda^{x_0-1}} &= \zeta+p(1-\zeta), \\
    \intertext{and hence}
p &= \frac{\zeta}{\zeta+\lambda^{x_0-1}}.
\end{align*}

There remains to evaluate the expectancy of~$\abs{\bI_H \cap N(v)}$, which is
\begin{align*}
\esp{\abs{N(v)\cap \bI_H} \mid E_J} &= \pr{v\notin \bI_0}\Big(x_0 + \esp{\abs{\bI_1} \mid v\notin \bI_0}\Big)
= \frac{\lambda^{x_0-1}}{\zeta+\lambda^{x_0-1}} \pth{x_0 + \sum_{j=1}^{d-1} \frac{x_j}{1+\lambda^{j-1}}}.
\end{align*}

\medskip

We now assume that~$x_0=0$. 
The probability distribution of~$\bI_H \cap W$ is obtained from that above by forbidding the outcomes that correspond to a
non-maximal independent set when~$x_0=0$.  There is only one such outcome, which corresponds to the event~$v\notin \bI_0$
and~$\bI_1 = \varnothing$ (because~$v$ is no longer covered by~$W_{1,0}$ in that case). This means that $\bI_H \cap N^2[v]=W_2$, hence that outcome has zero contribution to both $\pr{v\in \bI_H \mid E_J}$ and $\esp{|N(v)\cap \bI_H| \mid E_J}$; forbidding that outcome therefore increases those quantities. It follows that the previously computed values are lower bounds in the case $x_0=0$.
%

\medskip

We conclude that regardless of the value of $x_0$, we have
\begin{align}
\pr{v\in \bI_H \mid E_J} &\ge \dfrac{\zeta}{\zeta + \lambda^{x_0-1}}, \quad \text{and}\label{eq:occupancy1}\\
\esp{\abs{N(v)\cap \bI_H} \mid E_J} &\ge \dfrac{\lambda^{x_0-1}}{\zeta + \lambda^{x_0-1}}\left(x_0 + \sum_{j=1}^{d-1} \frac{x_j}{1+\lambda^{j-1}} \right).\label{eq:occupancy2}
\end{align}

%
%
%

\noindent
There remains to check that, when~$\lambda=4$, it holds that
\begin{equation}\label{eq:occupancy}
\frac{2^{k-3}+k}{k} \pr{v\in \bI_H \mid E_J} + \frac{2}{k}\esp{\abs{N(v)\cap \bI_H} \mid E_J} \geq 1.
\end{equation}
By combining~\eqref{eq:occupancy1},~\eqref{eq:occupancy2}, and~\eqref{eq:occupancy}, it is enough to prove that
\[\pth{2^{k-3}+k}\zeta  + 2\lambda^{x_0-1}\pth{x_0 + \sum_{j=1}^{d-1} \frac{x_j}{1+\lambda^{j-1}}} \geq k(\zeta + \lambda^{x_0-1}), \]
that is, multiplying both sides by the positive real~$\lambda^{1-x_0}/\zeta$,
\begin{equation}\label{eq-calculus}
    \pth{2^{k-3}+k}\lambda^{1-x_0}  + \frac2\zeta\pth{x_0 + \sum_{j=1}^{d-1} \frac{x_j}{1+\lambda^{j-1}}} \geq k\lambda^{1-x_0} + \frac{k}{\zeta}.
\end{equation}
Let us set 
%
\[ R\coloneqq {\zeta}^{-1}\cdot\left(k - 2x_0 - 2\sum_{j=1}^{d-1} \frac{x_j}{1+\lambda^{j-1}} \right) = \prod_{j=1}^{d-1} \left( 1 + \frac{1}{\lambda^{j-1}}\right)^{x_j}\left(k - 2x_0 - 2\sum_{j=1}^{d-1} \frac{x_j}{1+\lambda^{j-1}} \right),
\]
so that~\eqref{eq-calculus} is equivalent to 
\begin{equation}
\label{eq:R}
2^{k-3}\lambda^{1-x_0} \geq R.
\end{equation}

\noindent
Observe that the left side of~\eqref{eq:R} is positive, and recall that, by definition, each value~$x_j$ is a non-negative integer. We may therefore assume that~$x_0<\frac{k}{2}$, since otherwise~$R\le 0$, which directly implies~\eqref{eq:R}. Likewise,~$R\le 0$
if~$x_1 \ge k-2x_0$, hence we may moreover assume that~$0\le x_1 \le k-2x_0-1$.

Let us now fix~$\lambda=4$ and prove $\eqref{eq:R}$ in that case, that is~$R\le 2^{k-2x_0-1}$ for every non-negative integer~$x_0$. Since the inequality is verified if~$R\le 0$, we
assume from now on that~$R$ is positive.
We define~$y_j \coloneqq \dfrac{x_j}{1+\lambda^{j-1}}$ for every~$j\in\{1,\dotsc,d-1\}$, and~$y \coloneqq \sum_{j=2}^{d-1}\limits y_j$.
If~$y>0$ then we set $j_0 \coloneqq \min \{j\ge 2, x_j\neq 0\}$, and otherwise we set~$j_0\coloneqq\infty$. In particular,~$y\ge \frac{1}{1+\lambda^{j_0}}$.
We have

\begin{align*}
R &= 2^{x_1} \cdot \prod_{j=2}^{d-1} \left( 1 + \frac{1}{\lambda^{j-1}}\right)^{x_j}\left(k-2x_0-x_1 - 2\sum_{j=2}^{d-1} \frac{x_j}{1+\lambda^{j-1}} \right) \\
&= 2^{x_1} \cdot \prod_{j=j_0}^{d-1} \left( 1 + \frac{1}{\lambda^{j-1}}\right)^{(1+\lambda^{j-1})y_j}\left(k-2x_0-x_1 - 2\sum_{j=j_0}^{d-1} y_j\right)\\
&\le 2^{x_1} \cdot \prod_{j=j_0}^{d-1}  \left( 1 + \frac{1}{\lambda^{j_0-1}}\right)^{(1+\lambda^{j_0-1})y_j}\left(k-2x_0-x_1 - 2\sum_{j=j_0}^{d-1} y_j\right)\\
&= 2^{x_1}\left( 1 + \frac{1}{\lambda^{j_0-1}}\right)^{(1+\lambda^{j_0-1})y} (k-2x_0-x_1 - 2y),
\end{align*}
where the inequality uses that~$\displaystyle x\mapsto{\left(1+\frac1{\lambda^{x}}\right)}^{(1+\lambda^x)}$ is decreasing (to~$1$) as~$x$ increases in~$\mathbb{R}_{>0}$.

\noindent
We let~$A\coloneqq \left( 1 + \frac{1}{\lambda^{j_0-1}}\right)^{1+\lambda^{j_0-1}}$,~$B \coloneqq k-2x_0-x_1$, and~$f\colon x \mapsto A^x(B-2x)$, so that~$R \le 2^{x_1} f(y)$.
Since $j_0 \ge 2$, we have $A \le \frac{3125}{1024}$, and by the assumptions~$B$ is a positive integer. 

We are going to use the following fact in order to systematically obtain an upper bound on~$f(y)$.
\begin{description}
\item[Fact 1] For all real numbers~$y_0$,~$A$ and~$B$ with~$A>1$ and~$B>0$,
    the maximum of the function~$f\colon x \mapsto A^x(B-2x)$ on~$\mathbb{R}$
        is~$\frac{2A^{B/2}}{e\ln A}$, and if~$B/2-1/\ln A \le y_0$ then
        the maximum on the domain~$[y_0,+\infty)$ is~$f(y_0)$.
\end{description}

\noindent
There are now three cases.
\begin{enumerate}[label=(\roman*)]
    \item If~$B=1$, then~$\frac{B}{2}-\frac{1}{\ln A} < 0$, hence~$f(y)\le f(0) = 1$. So~$R\le 2^{x_1} = 2^{k-2x_0-B}= 2^{k-2x_0-1}$. 
\item If~$B=2$, we have
\begin{align*}
\frac{B}{2} - \frac{1}{\ln A} &= 1 - \frac{1}{(1+\lambda^{j_0-1})\ln(1+1/\lambda^{j_0-1})} \\
& \le 1 - \frac{1}{(1+\lambda^{j_0-1})\cdot 1/\lambda^{j_0-1}} = \frac{1}{1+\lambda^{j_0-1}} \le y.
\end{align*}
So $f(y) \le f\hspace{-3pt}\pth{\frac{1}{1+\lambda^{j_0-1}}} = 2$, and $R\le 2^{x_1+1} = 2^{k-2x_0-B+1} = 2^{k-2x_0-1}$.
\item Finally, if~$B \ge 3$, then since~$f(y) \le \frac{2A^{B/2}}{e \ln A}$ we deduce that~$R\le 2^{x_1}f(y) \le 2^{k-2x_0} \cdot \underbrace{\tfrac{2(A/4)^{B/2}}{e\ln A}}_{< 1/2} < 2^{k-2x_0-1}$.
\end{enumerate}

This finishes to prove that~$R\le2^{k-2x_0-1}$.
This establishes~\eqref{eq-calculus}, and therefore~\eqref{eq:occupancy}.
The conclusion follows by taking the convex combination of~\eqref{eq:occupancy} over all possible values of~$J$.
\end{proof}

We are now ready to prove \Cref{thm:girth7}.

\begin{proof}[Proof of \Cref{thm:girth7}]
We set~$\lambda\coloneqq4$, and apply \Cref{algo} with 
\begin{align*}
\alpha_v = 1 + \frac{2^{k(v)-3}}{k(v)} \quad \et \quad \beta_v = \frac{2}{k(v)}
\end{align*}
for every vertex~$v\in V(G)$, where~$k(v)$ is chosen such that
$\frac{2\deg(v)+2^{k-3}}{k}$ is minimised when~$k=k(v)$.
The results follows from \eqref{as-A}.
\end{proof}

\section{Bounds on the Hall ratio}\label{sec-hall} 
We focus on establishing upper bounds on the Hall ratios of graphs with bounded maximum
degree and girth. These bounds are obtained by using the uniform distribution
on~$\cI_{\alpha}(G)$, for~$G$ in the considered class of graphs, into
Lemma~\ref{ratio-vertex} or Lemma~\ref{ratio-edge}.

\subsection{Structural analysis of a neighbourhood} 
We start by introducing some terminology.
\begin{defi}\label{pattern}\mbox{}
\begin{enumerate}
\item A \emph{pattern of depth~$r$} is any graph~$P$ given with a root vertex~$v$ such that \[\forall u\in V(G),\quad \di_G(u,v) \leq r.\]
The \emph{layer at depth~$i$} of~$P$ is the set of vertices at distance~$i$ from its root vertex~$v$.
\item A pattern~$P$ of depth~$r$ and root~$v$ is \emph{$d$-regular} if all its vertices have degree exactly~$d$, except maybe in the two deepest layers where the vertices are only required to have degree at most~$d$.
    \end{enumerate}
\end{defi}
\begin{defi}
    Let~$P$ be a pattern with depth~$r$ and root~$v$. Let~$\bI$ be a uniform random maximum
    independent set of~$P$. We define~$e_i(P) \coloneqq \esp{\abs{\bI \cap N^i_P(v)}}$ for each~$i
    \in\{0,\dotsc,r\}$.
\begin{enumerate} 
\item The \emph{constraint} associated to~$P$ is the pair~$c(P)=(\be(P),n(P))$, where
~$\be(P)=(e_i(P))_{i=0}^r \in{\left(\mathbb{Q}^+\right)}^{r+1}$, and~$n(P)\	\in \mathbb{N}$ is the
    \emph{cardinality} of the constraint, which is the number of maximum independent sets of~$P$.
    Most of the time, we only need to know the value of~$\be(P)$, in which case we characterise the constraint~$c(P)=(\be(P),n(P))$ only by~$\be(P)$. The value of~$n(P)$ is only needed for a technical reason, in order to be able to compute constraints inductively.
\item Given two constraints~$\be, \be' \in
    \left(\mathbb{Q}^+\right)^{r+1}$, we say that~$\be$ is
    \emph{weaker} than~$\be'$ if, for any vector~$\ba \in
    \left(\mathbb{Q}^+\right)^{r+1}$ it holds that
\[ \ba^{\!\top} \be' \ge 1 \quad \implies \quad \ba^{\!\top} \be \ge 1.\]
    If the above condition holds only for all
    vectors~$\ba\in{\left(\mathbb{Q}^+\right)}^{r+1}$ with
    non-increasing coordinates, then we say that~$\be$ is
    \emph{relatively weaker} than~$\be'$.
\end{enumerate}
\end{defi}

\noindent
Note that~$\be$ is weaker than~$\be'$ if and only if $e_i \ge e'_i$ for every $i \in \{0,\dotsc,r\}$, and~$\be$ is relatively weaker than~$\be'$ if and only if $\sum_{j=0}^i \limits e_j \ge \sum_{j=0}^i \limits e'_j$ for every $i \in \{0,\dotsc,r\}$.

\begin{rk}\label{deg2}
Let~$P$ be a pattern such that one of its vertices~$u$ is adjacent with some
    vertices~$u_1,\dotsc, u_k$ of degree~$1$ in the next layer, where~$k\ge 2$. Then every
    maximum independent set of~$P$ contains~$\{u_1,\dotsc u_k\}$ and not~$u$.
    Consequently,~$\be(P)$ is weaker than~$\be(P\setminus \{u_3, \dotsc, u_k\})$ since,
    letting~$i$ be the distance between~$u_1$ and the root of~$P$, one has
\[
e_j(P) = \begin{cases}
      e_j(P\setminus \{u_3, \dotsc, u_k\}) & \quad \text{if~$j\neq i$, and} \\
  e_i(P\setminus \{u_3, \dotsc, u_k\}) + (k-2) & \quad \mbox{if j=i.}
\end{cases}
\]
\end{rk}

\subsection{Tree-like patterns} 
\subsubsection{Rooting at a vertex}\label{root-vertex} 
Fix an integer~$r \ge 2$. Let~$G$ be a~$d$-regular graph of girth at least~$2r+2$, and
let~$\bI$ be a uniform random maximum independent set of G. For any fixed vertex~$v$, we let~$J$
be any possible realisation of~$\bI\setminus N^r[v]$, and~$W\coloneqq N^r[v]\setminus N(J)$.
By the Spatial Markov Property of the uniform distribution over the maximum independent sets
of~$G$, the random independent set~$\bI\cap N^r[v]$ follows the uniform distribution over the
maximum independent sets of~$G[W]$. Now, observe that~$G[W]$ is a~$d$-regular pattern of
depth~$r$ with root vertex~$v$, and since~$G$ has girth at least~$2r+2$, this pattern is moreover a tree.
Let~$\mathcal{T}_r(d)$ be the set of acyclic~$d$-regular patterns of depth~$r$.

Let us define~$\bX_i(v) \coloneqq \bI \cap N^i(v)$, for every~$i\in\{0,\dotsc,r\}$.  We seek
parameters~$(\alpha_i)_{i=0}^{r}$ such that the inequality~$\sum_{i=0}^r \alpha_i
\esp{\abs{\bX_i(v)}} \ge 1$ is satisfied regardless of the choice of~$v$. To this end, it is
enough to pick the rational numbers~$\alpha_i$ in such a way that the inequality is
satisfied in any tree~$T\in \mathcal{T}_r(d)$, when~$v$ is the root vertex. In a more formal
way, given any~$T\in \mathcal{T}_r(d)$, the vector~$\ba = (\alpha_0, \dotsc, \alpha_r)$ must
be \emph{compatible} with the constraint~$\be(T)$, that is,~$\ba^{\!\top}\be(T) \ge 1$ for
each~$T\in \mathcal{T}_r(d)$.

An application of Lemma~\ref{ratio-vertex} then lets us conclude that $\abs{V(G)}/\alpha(G)$ is bounded from above by the solution to the following linear program.

\begin{align}\label{lp}
\text{Minimise} &\displaystyle \quad\quad \alpha_0 + \sum_{i=1}^r \alpha_i d(d-1)^{i-1}\\
\text{such that} & \begin{cases}
    \forall T\in \mathcal{T}_r(d),& \displaystyle\sum_{i=0}^r \alpha_i e_i(T) \ge 1 \\
    \forall i \le r, & \alpha_i \ge 0.
\end{cases}\notag
\end{align}

The end of the proof is made by computer generation of~$\mathcal{T}_r(d)$, in
order to generate the desired linear program, which is then solved again by computer
computation. For the sake of illustration, we give a complete human proof of
the case where~$r=2$ and~$d=3$. There are~$10$ trees in~$\mathcal{T}_2(3)$. One can easily
compute the constraint~$(e_0(T),e_1(T),e_2(T))$ for each~$T\in
\mathcal{T}_2(3)$; they are depicted in Figure~\ref{trees}. Note that
constraints~$\be_8$,~$\be_9$ and~$\be_{10}$ are weaker than
constraint~$\be_7$, so we may disregard these constraints in the linear
program to solve. Note also that constraint~$\be_0$ is relatively
weaker than constraint~$\be_1$, and so may be disregarded as well, provided that the solution of the linear program is attained by a
vector~$\ba$ with non-increasing coordinates, which will have to be
checked. The linear program to solve is therefore the following.

\begin{align*}
&\text{Minimise } \quad\quad\alpha_0 + 3\alpha_1 + 6\alpha_2 \\
&  \text{such that } \left\{\begin{aligned}
    &  5/2 \cdot \alpha_1  + 1/2 \cdot \alpha_2 &{}\ge{}& &1\\
    &  2 \alpha_1 + 2\alpha_2 &{}\ge{}& &1 \\
    & 1/5 \cdot \alpha_0 + 8/5 \cdot \alpha_1 + 6/5 \cdot \alpha_2 &{}\ge{}& &1 \\
    & 1/3 \cdot \alpha_0 + \alpha_1 + 8/3\cdot \alpha_2 &{}\ge{}&  &1 \\
    & 1/2\cdot \alpha_0 + 1/2\cdot \alpha_1 + 4\alpha_2 &{}\ge{}& &1 \\
    & \alpha_0 + 3\alpha_2 &{}\ge{}&  &1 \\
    & \alpha_0, \alpha_1, \alpha_2 \ge 0. & &
\end{aligned}\right.
\end{align*}

The solution of this linear program is~$\frac{85}{31}\approx 2.741935$, attained
by~$\ba = \left(\frac{19}{31},\frac{14}{31},\frac{4}{31} \right)$,
which indeed has non-increasing coordinates. This is an upper bound
on~$\rho(3,6)$, though we prove a stronger one through a more involved computation in Section~\ref{sec:edge-root}.

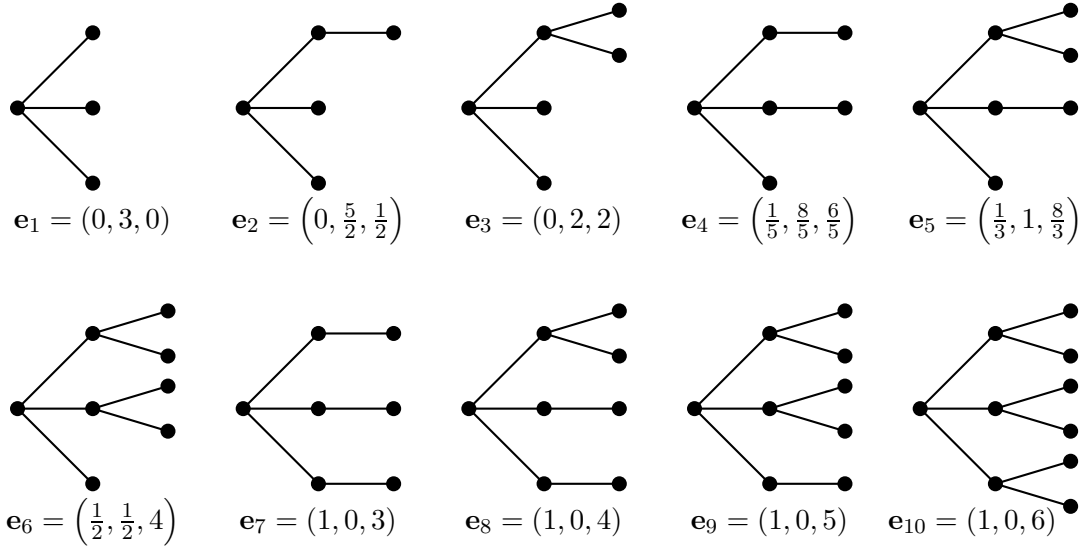
\begin{figure}[!ht]
\begin{center}
\begin{tikzpicture}[scale=0.75]

\foreach \a in {0,3.5,...,15}{
	\draw[thick] (\a,0) -- (\a+1,0);
	\draw[thick] (\a,0) -- (\a+1,1);
	\draw[thick] (\a,0) -- (\a+1,-1);
	\fill (\a,0) circle(0.1);
	\fill (\a+1,1) circle(0.1);
	\fill (\a+1,0) circle(0.1);
	\fill (\a+1,-1) circle(0.1);
	}
	
\foreach \a in {0,3.5,...,15}{
	\draw[thick] (\a,4) -- (\a+1,4);
	\draw[thick] (\a,4) -- (\a+1,5);
	\draw[thick] (\a,4) -- (\a+1,3);
	\fill (\a,4) circle(0.1);
	\fill (\a+1,5) circle(0.1);
	\fill (\a+1,4) circle(0.1);
	\fill (\a+1,3) circle(0.1);
	}
	
\node at (1,2.2) {$\mathbf{e}_1=(0,3,0)$};	
	
\draw[thick] (4.5,5) -- (5.5,5);
\fill (5.5,5) circle(0.1);
\node at (4.5,2.2) {$\mathbf{e}_2=\left(0,\frac 52,\frac 12\right)$};

\draw[thick] (8,5) -- (9,5.3);
\fill (9,5.3) circle(0.1);
\draw[thick] (8,5) -- (9,4.7);
\fill (9,4.7) circle(0.1);
\node at (8,2.2) {$\mathbf{e}_3=(0,2,2)$};

\draw[thick] (11.5,5) -- (12.5,5);
\fill (12.5,5) circle(0.1);
\draw[thick] (11.5,4) -- (12.5,4);
\fill (12.5,4) circle(0.1);
\node at (11.5,2.2) {$\mathbf{e}_4=\left(\frac{1}{5},\frac{8}{5},\frac{6}{5}\right)$};

\draw[thick] (15,5) -- (16,5.3);
\draw[thick] (15,5) -- (16,4.7);
\fill (16,5.3) circle(0.1);
\fill (16,4.7) circle(0.1);
\draw[thick] (15,4) -- (16,4);
\fill (16,4) circle(0.1);
\node at (15,2.2) {$\mathbf{e}_5=\left(\frac 13,1,\frac 83\right)$};

\draw[thick] (1,1) -- (2,1.3);
\draw[thick] (1,1) -- (2,0.7);
\fill (2,1.3) circle(0.1);
\fill (2,0.7) circle(0.1);
\draw[thick] (1,0) -- (2,0.3);
\draw[thick] (1,0) -- (2,-0.3);
\fill (2,0.3) circle(0.1);
\fill (2,-0.3) circle(0.1);
\node at (1,-1.8) {$\mathbf{e}_6=\left( \frac 12,\frac 12,4\right)$};

\draw[thick] (4.5,1) -- (5.5,1);
\fill (5.5,1) circle(0.1);
\draw[thick] (4.5,0) -- (5.5,0);
\fill (5.5,0) circle(0.1);
\draw[thick] (4.5,-1) -- (5.5,-1);
\fill (5.5,-1 ) circle(0.1);
\node at (4.5,-1.8) {$\mathbf{e}_7=(1,0,3)$};

\draw[thick] (8,1) -- (9,1.3);
\draw[thick] (8,1) -- (9,0.7);
\fill (9,1.3) circle(0.1);
\fill (9,0.7) circle(0.1);
\draw[thick] (8,0) -- (9,0);
\fill (9,0) circle(0.1);
\draw[thick] (8,-1) -- (9,-1);
\fill (9,-1) circle(0.1);
\node at (8,-1.8) {$\mathbf{e}_8=(1,0,4)$};

\draw[thick] (11.5,1) -- (12.5,1.3);
\draw[thick] (11.5,1) -- (12.5,0.7);
\fill (12.5,1.3) circle(0.1);
\fill (12.5,0.7) circle(0.1);
\draw[thick] (11.5,0) -- (12.5,0.3);
\draw[thick] (11.5,0) -- (12.5,-0.3);
\fill (12.5,0.3) circle(0.1);
\fill (12.5,-0.3) circle(0.1);
\draw[thick] (11.5,-1) -- (12.5,-1);
\fill (12.5,-1) circle(0.1);
\node at (11.5,-1.8) {$\mathbf{e}_9=(1,0,5)$};

\draw[thick] (15,1) -- (16,1.3);
\draw[thick] (15,1) -- (16,0.7);
\fill (16,1.3) circle(0.1);
\fill (16,0.7) circle(0.1);
\draw[thick] (15,0) -- (16,0.3);
\draw[thick] (15,0) -- (16,-0.3);
\fill (16,0.3) circle(0.1);
\fill (16,-0.3) circle(0.1);
\draw[thick] (15,-1) -- (16,-0.7);
\draw[thick] (15,-1) -- (16,-1.3);
\fill (16,-0.7) circle(0.1);
\fill (16,-1.3) circle(0.1);
\node at (15,-1.8) {$\mathbf{e}_{10}=(1,0,6)$};

\end{tikzpicture}
\vspace{-6pt}
\caption{An enumeration of~$\be(T)$ for all trees~$T\in \mathcal{T}_2(3)$.}\label{trees}
\end{center}
\end{figure}

To compute~$\be(T)$ for each~$T\in \mathcal{T}_r(d)$, one can
enumerate all the maximum independent sets of~$T$ and average
the size of their intersection with each layer of~$T$. For general graphs,
there might be no better way of doing so, however the case
of~$\mathcal{T}_r(d)$ can be treated inductively by a standard approach: we
distinguish between the maximum independent sets that contain the root and
those that do not. To this end,
we slightly extend our notion of ``constraint'' to any pair~$(\be,n)$
where~$\be\in{\left(\mathbb{Q}^+\right)}^{r+1}$ and~$n$ is a non-negative integer;
the constraints we shall use have a combinatorial interpretation
with respect to some pattern.

\begin{defi}\label{def-cjoin}
Let~$c=(\be,n)$ and~$c'=(\be',n')$ be two constraints.
\begin{enumerate}
\item The operation~$\vee$ on~$c$ and~$c'$ returns the constraint
\[ c \vee c' \coloneqq \begin{cases}
    \pth{\dfrac{n}{n+n'} \be + \dfrac{n'}{n+n'}\be', n+n'}
    & \quad\text{if }\lVert\be\rVert_1=\lVert\be'\rVert_1,\\
    c & \quad\text{if }\lVert\be\rVert_1>\lVert\be'\rVert_1,\\
    c' & \quad\text{if }\lVert\be\rVert_1<\lVert\be'\rVert_1.
\end{cases}\]
\item 
The operation~$\oplus$ on~$c$ and~$c'$ returns the
constraint~$c \oplus c' \coloneqq \big(\be + \be', n\cdot n' \big)$.
\end{enumerate}
\end{defi}

For a given tree~$T\in \mathcal{T}_r(d)$ with root~$v$, let~$c_1(T)$ be the constraint
associated to~$T$ where~$v$ is forced (that is, we restrict to the maximum independent sets that
contain~$v$ when computing the constraint~$c_1(T)$), and let~$c_0(T)$ be the constraint
associated to~$T$ where~$v$ is forbidden.  It readily follows from
Definition~\ref{def-cjoin} that
\[c(T) = c_0(T) \vee c_1(T).\]

If~$(T_i)_{i=1}^{d}$ are the (possibly empty) subtrees of~$T$ rooted at the children of the
root~$v$, then

\vspace{-12pt}
\begin{align*}
    c_0(T) &= \big( (0,\be), n \big) && \mbox{ where } (\be,n) = \bigoplus_{i=1}^{d} c(T_i), \quad\text{and}\\
    c_1(T) &= \big( (1,\be), n \big) && \mbox{ where } (\be,n) = \bigoplus_{i=1}^{d} c_0(T_i).
\end{align*}
We thus obtain an inductive way of computing~$\be(T)$ by using the following initial values.

\begin{align*}
c_0(\varnothing) &\coloneqq \big((0),1\big) & c_1(\varnothing) &\coloneqq \big((0),0\big) \\
c_0(\{v\}) &\coloneqq \big((0),1\big) & c_1(\{v\}) &\coloneqq \big((1),1\big).\\
\end{align*}

Using this inductive way to enumerate the vectors~$\be(T)$ for~$T\in
\mathcal{T}_r(d)$, the following statement is obtained by computer calculus.

\begin{lemma}\label{vertex-bounds}
The solution to the linear program~\eqref{lp} is
\begingroup
\allowdisplaybreaks
\begin{align*}
\mathcal{T}_3(3) \colon \quad& \frac{5849}{2228} \approx 2.625224 && \text{with } \ba = \left(\frac{953}{2228}, \frac{162}{557}, \frac{81}{557}, \frac{21}{557} \right),\\
\mathcal{T}_4(3) \colon \quad& \frac{2098873192}{820777797} \approx 2.557176 &&\text{with }  \ba = \left(\frac{225822361}{820777797},\frac{18575757}{91197533},\frac{10597368}{91197533},\right.&\\*
&&&\multicolumn{1}{r}{$\left.\dfrac{5054976}{91197533},\dfrac{1172732}{91197533}\right),$}\\
\mathcal{T}_5(3) \colon\quad& \frac{29727802051155412}{11841961450578397} \approx 2.510378  && \multicolumn{2}{l}{with~$\ba = \left(\dfrac{3027359065168972}{11841961450578397},\dfrac{2216425114872980}{11841961450578397},\right.$}\\*
& \multicolumn{4}{r}{$\left.\dfrac{2224040336719575}{23683922901156794},\dfrac{2026654050681425}{47367845802313588},\dfrac{403660478424775}{23683922901156794},\dfrac{51149140376400}{11841961450578397}\right),$} \\
\mathcal{T}_3(4) \colon\quad& \frac{7083927}{2331392} \approx 3.038497 &&\multicolumn{2}{l}{with~$\ba =  \left(\dfrac{123345}{333056},\dfrac{68295}{291424},\dfrac{12283}{145712},\dfrac{2911}{145712}\right),$} \\
\mathcal{T}_4(4) \colon\quad& 3  && \text{with } \ba = \left(\frac{7}{43},\frac{6}{43}, \frac{19}{258}, \frac{7}{258}, \frac{1}{258} \right),\\ 
\mathcal{T}_2(5) \colon\quad& \frac{69}{19} \approx 3.631579 &&\text{with }  \ba = \left(\frac{37}{57}, \frac{6}{19}, \frac{4}{57}\right), \\
\mathcal{T}_3(5) \colon\quad& \frac{7}{2} = 3.5 && \text{with }  \ba = \left(\frac{77}{282},\frac{25}{141},\frac{17}{282},\frac{2}{141}\right).
\end{align*}
\endgroup
\end{lemma}

\subsubsection{Rooting in an edge}\label{sub-edge-rooted} 
Definition~\ref{pattern} can be extended to a pattern with a root-edge instead of a
root-vertex. The distance in a pattern~$P$ between a vertex~$w$ and an edge~$uv$ is defined
to be~$\min\{\di_P(w,u),\di_P(w,v)\}$. The depth of a pattern~$P$ rooted in an edge~$e$ is
then the largest distance between~$e$ and a vertex in~$P$.
It is possible to follow the same analysis as in Section~\ref{root-vertex} with edge-rooted
patterns: in order for the edge-rooted pattern of depth~$r$ to always be a tree, the
graph~$G$ must have girth at least~$2r+3$. Let~$\mathcal{T}'_r(d)$ be the set of acyclic
edge-rooted~$d$-regular patterns of depth~$r$. By \Cref{ratio-edge}, the linear program
to solve is now the following.

\begin{align}\label{lp-edge}
    \frac{\abs{V(G)}}{\alpha(G)} &\le  \displaystyle\min \quad 2 \sum_{i=0}^r \alpha_i (d-1)^i\\
\text{such that} & \begin{cases}
    \forall T\in \mathcal{T}'_r(d),& \displaystyle\sum_{i=0}^r \alpha_i e_i(T) \ge 1 \\
    \forall i \le r, & \alpha_i \ge 0.
\end{cases}\notag
\end{align}

For a given tree~$T \in \mathcal{T}'_r(d)$ rooted in~$e=uv$, it is
possible to compute~$\be(T)$ using the constraints associated to
vertex-rooted trees. If~$T_u$ and~$T_v$ are the subtrees of~$T$ respectively rooted at~$u$
and at~$v$, then it readily follows from Definition~\ref{def-cjoin} that
\begin{equation}\label{eq-prod}
c(T) = \Big(c_0(T_u)\oplus c_0(T_v)\Big) \vee
\Big(c_0(T_u)\oplus c_1(T_v)\Big) \vee
\Big(c_1(T_u)\oplus c_0(T_v)\Big).
\end{equation}

Following the enumeration of the vectors~$\be(T)$ for~$T\in
\mathcal{T}'_r(d)$ described earlier, the next statement is obtained by
computer calculus.

\begin{lemma}\label{edge-bounds}
The solution to the linear program~\eqref{lp-edge} is
\begin{align*}
\mathcal{T}'_2(3)\colon&\quad \frac{30}{11} \approx 2.72727 && \text{with } \ba {}={} \pth{\frac{1}{2},\frac{13}{44},\frac{3}{44}},\\
\mathcal{T}'_3(3)\colon&\quad \frac{125}{48} \approx 2.604167	&& \text{with }  \ba {}={} \left(\frac{11}{32},\frac{5}{24},\frac{3}{32}, \frac{1}{48}\right),\\
\mathcal{T}'_4(3)\colon&\quad \frac{14147193}{5571665} \approx 2.539132	&&\text{with }  \ba {}={} \left(\frac{98057}{506515},\frac{159348}{1114333},\frac{3688469}{44573320}, \frac{1752117}{44573320}, \frac{402569}{44573320}\right),\\
\mathcal{T}'_2(4)\colon&\quad \frac{41}{13} \approx 3.153846  && \text{with }  \ba {}={} \left(\frac{11}{26},\frac{3}{13},\frac{2}{39}\right),\\
\mathcal{T}'_3(4)\colon&\quad \frac{127937}{42400} \approx 3.017382  && \text{with }  \ba {}={} \left(\frac{5539}{16960}, \frac{1737}{10600}, \frac{257}{5300}, \frac{399}{42400}\right),\\
\mathcal{T}'_2(5)\colon&\quad \frac{18}{5} = 3.6 && \text{with }  \ba {}={} \left(\frac{17}{45},\frac{8}{45},\frac{2}{45}\right).
\end{align*}
\end{lemma}

The bounds obtained in Lemma~\ref{edge-bounds} are valid for graphs of girth at
least~$2r+3$. It turns out that the same bounds, with the
same~$\ba$, remain valid for graphs of girth~$2r+2=6$, when~$r=2$
and~$d \in \{3,4\}$. We were not able to check this for higher values of~$r$
or~$d$, but we propose the following conjecture which would explain and
generalise this phenomenon.

\begin{conj}\label{conj}
Let~$P$ be a~$d$-regular edge-rooted pattern of depth~$r$ and of girth~$2r+2$.
Then the constraint~$\be(P)$ is weaker than some convex
combination of constraints~$\be(T)$ with~$T\in \mathcal{T}'_r(d)$.
More formally, there exist~$T_1, \dotsc, T_m \in \mathcal{T}'_r(d)$ and
$\lambda_1, \dotsc, \lambda_m \in [0,1]$ with~$\sum_{i=1}^m \lambda_i = 1$
such that for any~$\ba\in{\left(\mathbb{Q}^+\right)}^{r+1}$,
\[ \ba^{\!\top}\left(\sum_{i=1}^m
\lambda_i \be(T_i) \right) \ge 1 \quad \implies \quad
\ba^{\!\top} \be(P) \ge 1.\]
\end{conj}

\vspace{6pt}
\subsection{More complicated patterns} 
\subsubsection{Rooting at a vertex} 
Let us fix a depth~$r \ge 2$. Let~$G$ be a~$d$-regular graph of girth~$g
\le 2r+1$. We repeat the same analysis as in Section~\ref{root-vertex}: we end
up having to find a vector~$\ba \in \mathbb{Q}^{r+1}$
compatible with all the constraints generated by vertex-rooted~$d$-regular
patterns of depth~$r$ and girth~$g$. Letting~$\mathcal{P}_r(d,g)$ be the
set of such patterns, we thus want that
\[ \forall P \in \mathcal{P}_r(d,g), \quad
\ba^{\!\!\top} \be(P) \ge 1. \]

In this setting, we could do no better than performing an exhaustive
enumeration of every possible pattern~$P \in \mathcal{P}_r(d,g)$, and computing the associated
constraint~$\be(P)$ through an exhaustive enumeration of
$\cI_{\alpha}(P)$. The complexity of such a process grows fast,
and we considered only depth~$r\le2$ and degree~$d\le 4$.
Since the largest value of the Hall ratio over the class
of~$3$-regular graphs of girth~$4$ or~$5$ is known to be~$\frac{14}{5}=2.8$, and the one of~$4$-regular
graphs of girth~$4$ is known to be~$\frac{13}{4}=3.25$, the only open value in
these settings is for the class of~$4$-regular graphs of girth~$5$.
Unfortunately, this method is not powerful enough to prove an upper bound lower
than~$\frac{13}{4}$, the obtained bound for~$\mathcal{P}_2(4,5)$
being~$\frac{82}{25} = 3.28$. It is more interesting to root the patterns in an edge.

\subsubsection{Rooting in an edge}\label{sec:edge-root} 
Similarly, we define~$\mathcal{P}'_r(d,g)$ to be the set of edge-rooted
$d$-regular patterns of girth~$g$. For fixed~$r$ and~$g$, we seek for the
solution of the following linear program.

\begin{align}\label{lp-edge-6}
    \frac{\abs{V(G)}}{\alpha(G)} &\le  \min \quad 2 \sum_{i=0}^r \alpha_i (d-1)^i\\
\text{such that} & \begin{cases}
    \forall P\in \mathcal{P}'_r(d,g),& \displaystyle\sum_{i=0}^r \alpha_i e_i(P) \ge 1 \\
    \forall i \le r, & \alpha_i \ge 0.
\end{cases}\notag
\end{align}

Again, our computations were limited to the cases where~$r\le2$ and~$d\le 4$.
However, we managed to prove improved bounds for girth~$6$ when~$d\in \{3,4\}$,
which seems to support Conjecture~\ref{conj}.

\begin{lemma}\label{lem:program-girth6}
The solution to the linear program~\eqref{lp-edge-6} is
\begin{equation*}
    \begin{aligned}[c]
\mathcal{P}'_2(3,6)\colon\qquad & \frac{30}{11} \approx 2.72727\\
\mathcal{P}'_2(4,6)\colon\qquad & \frac{41}{13} \approx 3.153846\\
\end{aligned}
\qquad
    \begin{aligned}[c]
        \text{with }  \ba &{}={} \left(\frac{1}{2},\frac{13}{44},\frac{3}{44}\right),\\
        \text{with }  \ba &{}={} \left(\frac{11}{26},\frac{3}{13},\frac{2}{39}\right).
\end{aligned}
\end{equation*}
\end{lemma}

\subsection{Proof of \texorpdfstring{\Cref{thm:ratio}}{Theorem 11}}

We are now ready to give a proof of the values presented in \Cref{tab:4}.

\begin{proof}[Proof of \Cref{thm:ratio}]
We want to establish the upper bounds on~$\rho(d,g)$ presented in \Cref{tab:4}. They all follow from an application of \Cref{ratio-vertex}
or \Cref{ratio-edge} using the solution of the appropriate linear program.
\begin{itemize}
    \item The upper bounds for~$\rho(3,6)$ and~$\rho(4,6)$ follow from \Cref{lem:program-girth6} by applying \Cref{ratio-edge}.
    \item The upper bounds for~$\rho(5,6)$, for~$\rho(d,8)$ with~$d\in\{3,4,5\}$, for~$\rho(d,10)$ with~$d\in\{3,4\}$
        and for~$\rho(3,12)$ follow from \Cref{vertex-bounds} by applying \Cref{ratio-vertex}.
    \item The upper bounds for~$\rho(d,7)$ with~$d\in\{3,4,5\}$, for~$\rho(d,9)$ with~$d\in\{3,4\}$ and for~$\rho(3,11)$
        follow from \Cref{edge-bounds} by applying \Cref{ratio-edge}.
    \qedhere
\end{itemize}
\end{proof}

\section{Conclusion}
We finish with a discussion about our method, its possible future applications, and its limitations.

\subsection{Going further}
We have decided to restrict ourselves in this work to graphs with given girth. Our general framework reduces the remaining work to an analysis of the behaviour of the hard-core distribution in trees of a given depth, which minimises the need to use massive computations through a computer in order to derive our results.
In the future, it will be interesting to apply our framework to other classes of graphs. We can imagine any class of graphs which is locally constrained, that is graphs of maximum degree~$d$ whose neighbourhoods up to a fixed depth~$r$ is a strict subset of graphs of maximum degree~$d$ and radius~$r$. For instance, one could consider the Hall ratio and fractional chromatic number of~$C_\ell$-free graphs for a fixed value of~$\ell$, of~$K_4$-free graphs, or of squares of graphs of girth at least some constant~$g$.

It would be interesting to be able to feed our framework with other probability distributions on the independent sets. A promising one would be the following: let~$\bI$ be a random maximal independent set obtained greedily by fixing a uniform random priority ordering on the vertices. Shearer~\cite{She83}  observed that~$\esp{\bI}$ matches the bound given in \Cref{thm:shearer}. However, such a distribution has a strong global dependency; it might be really difficult --- if possible at all --- but probably fruitful, to find a spatial Markov property satisfied by this distribution.

\subsection{On the hard-core model}
In the present work, we have fed our method with the hard-core distribution at different regimes. We have considered all the independent sets with a fugacity~$\lambda$ tending to~$0$, restricted ourselves to the maximal independent sets with a fugacity~$\lambda$ equal to~$4$, and then to the maximum independent sets, which forces the uniform distribution with~$\lambda=\infty$.

We observe that when we let the fugacity~$\lambda$ grow, and hence favour larger independent sets, we obtain better bounds for really small values of the degree. On the other hand, we are more constrained in our choices, which translates into worse asymptotic bounds. 

It could be unsettling that restricting to maximal independent sets does not yield a strict improvement in the bound obtained with our method; indeed it is always possible to find an optimal fractional colouring using only maximal independent sets as colour classes, up to covering some vertices with weight more that~$1$. However, because we fix the realisation of the random independent set on a large part of the graph, we suffer from a deeper propagation of the constraints due to the maximality of the independent set. When working with a partial realisation of a maximal independent set, vertices at distance~$2$ from that realisation can be forbidden, hence the need to work with deeper patterns (we must always have the freedom to pick the root of the pattern in the independent set). 

Another unsettling observation comes from the fact that the optimised fugacity when working with the whole set of independent sets is~$\lambda$ tending to~$0$ as the degree~$d$ tends to infinity. This means that the independent set that receives the largest weight in our fractional colouring is the empty independent set, which seems rather wasteful. This is however needed in order to assign a non negligible weight to the independent set that contains only the root (when we work with patterns of depth 1). Indeed, there is only one independent set that contains the root, while there are exponentially many independent sets that do not contain it. This is no longer the case when we restrict to maximal independent sets, which explains why the optimal value of the fugacity is a constant in this case.

\subsection{Limitations of the method}
While our framework gives a lot of freedom in its possible applications, it still suffers from some limitations which we discuss hereafter.

We cannot improve the bounds for the fractional chromatic number by trying to increase the girth. Indeed, since the class of patterns that we need to consider in our framework must be hereditary in that case, increasing the depth would only strictly increase the number of patterns (the ones of smaller depth are still present). Therefore, the girth of the class of graphs for which we can derive our bound is determined by the choice of the probability distribution, and the depth of its dependencies.

Working with maximum independent sets, it appears that our method cannot prove a better bound for~$\rho(d,\infty)$ than~$d/2 + 1$, which is reached at some girth~$g_0$. We could not find an explanation for this boundary, although such an explanation would have a strong theoretical interest. So, in order to obtain new bounds with our method for~$\rho(d,g)$, one would have to use the hard-core model in different regimes than the one where~$\lambda$ is infinite.

The bounds obtained appear not to be tight, if we consider the pairs~$(d,g)$ for which we know the value of~$\rho$ and~$\chi_f$ and compare it with the bound obtained using our method. However, obtaining these tight bounds systematically requires a deep structural analysis already when~$d=3$, and an even deeper one when~$d=4$. It is unlikely that these structural analyses could be generalised in order to cover larger values of~$d$, while our method provides a smooth transition between the small values of~$d$ and the asymptotic regime. 

\subsection{Note added}
After this work first appeared in a public preprint repository, a work from Cames van Batenburg, Goedgebeur, Joret~\cite{CGJ20} improved our bound for~$\rho(3,6)$, by establishing that it is at most~$8/3$. This work relies again on a deep structural analysis of cubic graphs avoiding some finite family of graphs (all of which contain a~$C_5$) as a subgraph.

\subsection{Acknowledgement}
We are thankful to Ewan Davies for the insightful discussion about the behaviour of the hard-core distribution
on closed neighbourhoods, which led to \Cref{thm:triangle-free}.  The authors sincerely thank both referees
for their careful reading of this work and suggestions for improvement.


\end{document}